\documentclass{article}
\usepackage[utf8]{inputenc}
\usepackage{amsthm}
\usepackage{amsmath}
\usepackage{comment}
\usepackage{amsfonts}
\usepackage{color}
\usepackage{latexsym}
\usepackage{geometry}
\usepackage{enumerate}
\usepackage[shortlabels]{enumitem}
\usepackage{csquotes}
\usepackage{graphicx}
\usepackage{comment}
\usepackage{appendix}
\usepackage{hyperref}
\usepackage{esint}
\usepackage{bm}
\usepackage{amssymb}

\hypersetup{
    colorlinks=true,
    linkcolor=blue,
    filecolor=magenta,      
    urlcolor=cyan,
}

\emergencystretch=\maxdimen
\def \bF {\mathbb{F}}
\def \bN {\mathbb{N}}
\def \bR {\mathbb{R}}

\def \fC {\mathfrak{C}}
\def \ft {\mathfrak{t}}
\def \fg {\mathfrak{g}}

\def \cA {\mathcal{A}}
\def \cC {\mathcal{C}}
\def \cS {\mathcal{S}}

\def \cR {\mathcal{R}}
\def \cX {\mathcal{X}}
\def \cN {\mathcal{N}}
\def \cW {\mathcal{W}}

\def \lt {\left}
\def \rt {\right}
\def \ep {\varepsilon}
\def \dt {\delta}

\DeclareMathOperator{\Rm}{Rm}
\DeclareMathOperator{\Ric}{Ric}

\DeclareMathOperator{\Var}{Var}
\DeclareMathOperator{\AVR}{AVR}
\DeclareMathOperator{\loc}{loc}
\DeclareMathOperator{\supp}{supp}

\makeatletter
\newcommand*{\rom}[1]{\rm {\expandafter\@slowromancap\romannumeral #1@}}
\makeatother

\def\XXint#1#2#3{{\setbox0=\hbox{$#1{#2#3}{\int}$ }
\vcenter{\hbox{$#2#3$ }}\kern-.6\wd0}}

\protected\def\vts{%
  \ifmmode
    \mskip0.5\thinmuskip
  \else
    \ifhmode
      \kern0.08334em
    \fi
  \fi
}

\numberwithin{equation}{section}
\newtheorem{Theorem}{Theorem}[section]
\newtheorem{Proposition}[Theorem]{Proposition}
\newtheorem{Lemma}[Theorem]{Lemma}

\theoremstyle{definition}

\newtheorem{Remark}[Theorem]{Remark}

\title{A gap theorem for metric solitons and its applications}
\author{Ganqi Wang, Yongjia Zhang\footnote{Yongjia Zhang's research is  supported by National Natural Science Foundation of China NSFC12301076.}}
\date{}

\begin{document}
\maketitle
\begin{abstract}
In this paper, we prove a gap theorem for $\bF$-limit metric solitons with respect to the asymptotic volume ratio (AVR): if the AVR of a metric soliton is sufficiently close to 1, then the metric soliton is Euclidean; this is a metric-soliton counterpart of \cite{WW25}. Our result can be applied to Ricci flows to derive a gap theorem and an $\varepsilon$-regularity theorem: (1) an ancient Ricci flow with a type-I scalar curvature bound and AVR close enough to 1 must be the static Euclidean space,  (2) a Ricci flow with locally type-I scalar curvature bound and local volume ratio close enough to 1 must be regular enough locally (in the sense that its curvature radius cannot be too small).

\end{abstract}
\section{Introduction}
An $n$-dimensional noncollapsed $\mathbb{F}$-limit metric soliton (metric soliton for short) $\left(\mathcal{X},(\nu_t)_{t\in(-T,0)}\right)$  arises as an $\mathbb{F}$-limit of a sequence of $n$-dimensional Ricci flows with bounded Nash entropy. It shares similar properties to a Ricci shrinker, among which self-similarity is an important one. Precisely, a metric soliton is modeled on some metric measure space $(X,d,\nu)$, called its \emph{model}, which is a metric completion of a smooth (Riemannian) metric measure space (called its \emph{regular part}) denoted by $(\mathcal{R}_X,\fg,f_0)$, where the smooth function $f_0:\cR_X\to\bR$ is called \emph{the potential function};  because of the self-similarity, sometimes we do not distinguish a metric soliton from its model. On $\cR_X$, we have
\begin{gather*}
    d\nu=(4\pi)^{-\frac{n}{2}}e^{-f_0}d\fg,\\
    \Ric+\nabla^2f_0=\frac{1}{2}\fg.
\end{gather*}
Thus, it is easy to see that a metric soliton is a Ricci shrinker with singularities. Recall that a shrinker $(M^n,g,f)$ is a tuple of a complete and smooth Riemannian manifold $(M^n,g)$ and a potential function $f$ satisfying the above equations on the whole manifold. For more details about metric solitons, see \cite{Bam20b,Bam23}, or \S 2.2.

In this paper, we consider gap theorems of metric solitons. Under certain conditions, we expect a metric soliton to be a Gaussian soliton, that is, the Euclidean space with canonical metric and $|x|^2/4$ as the potential function.
 
Many studies have been conducted on the rigidity for Gaussian shrinkers. For example, Yokota \cite{Yo12} proved  that if the normalized $f$-volume of a Ricci shrinker is close enough to $1$, then the Ricci shrinker is Gaussian. This, in light of \cite{CN09}, is equivalent to saying that, if Perelman's $\mu$-functional (see \cite{Per02}) is close enough to $0$, then the Ricci shrinker is Gaussian. Chan-Ma-Zhang \cite{CMZ25} and Chan-Zhang \cite{CZ26} proved similar gap theorems for the local $\mu$-functional on a geodesic ball. In particular, the local gap theorem \cite{CMZ25} was generalized to metric solitons in \cite{CMZ24}.

Li-Wang \cite{LW20} researched the rigidity of the Gaussian shrinker from another perspective. They proved that a normalized shrinker must be isometric to the Gaussian shrinker if it is close enough to the Gaussian shrinker in the Gromov-Hausdorff sense. For other rigidity results of shrinkers, especially of cylindrical shrinkers, see \cite{LW24, CM25, LZ23, Zhang18, Zhang20}.

Wang-Wang \cite{WW25} established a rigidity theorem for Ricci shrinkers in terms of the volume ratio on a large ball. 
\begin{Theorem}[{\cite[Theorem 1.2]{WW25}}]\label{thm WW25}
    There exists an $\ep=\ep(n)>0$ such that any complete Ricci shrinker $(M^n,g,f)$ with 
\begin{align*}
    \omega_n\ep^n|B_g(x,\ep^{-1})|\ge1-\ep
\end{align*}
is Gaussian. Here and henceforth, $\omega_n$ stands for the volume of the $n$-dimensional Euclidean unit ball.
\end{Theorem}
In this paper we prove a similar result for metric solitons.  Note that the AVR of a metric soliton does exist by \cite[Corollary 3.3]{DWZ26}. For a metric soliton $(X,d,\nu)$, we define the AVR to be
\begin{align*}
    \cA:=\lim_{A\to+\infty}\frac{|B_{\fg}(x_0,A)\cap\cR_X|}{\omega_nA^n},
\end{align*}
where $(\cR_X,\fg,f_0)$ is the regular part and $x_0$ is a center (see \S 2.2.2) of the metric soliton. 
\begin{Theorem}\label{thm-AVR gap for metric soliton}
There exists a $\dt=\dt(n)>0$ such that any $n$-dimensional metric soliton with   $\cA\ge1-\dt$ is the Gaussian soliton.
\end{Theorem}
The proof is similar to \cite{WW25}. By studying the monotonicity of 
$$h(\tau):=\int_{\cR_{X}}(4\pi\tau)^{-\frac{n}{2}}e^{-\frac{f_0-W}{\tau}}d\fg,\,\,\tau>0,$$ 
we conclude that $e^W=h(1)\ge\lim_{\tau\to+\infty}h(\tau)=\cA$, where $W\le 0$ is the soliton entropy (see \S 2.2.1). The crucial technical part is to verify the validity of the following integration by parts, namely, 
\begin{align*}
    \int_{\cR_{X}}\Delta f_0  e^{-\frac{f_0-W}{\tau}}d\fg=\int_{\cR_{X}} \frac{|\nabla f_0|^2}{\tau} e^{-\frac{f_0-W}{\tau}}d\fg.
\end{align*}
The proof of this equation is slightly different from the Ricci soliton case, since $\partial\cR_X$ contains points not at infinity. Once the above is established, the $\AVR$ gap theorem follows from the soliton entropy gap theorem in  \cite{CMZ24}.

The following quantitative version of Theorem \ref{thm-AVR gap for metric soliton} is more versatile in application.

\begin{Theorem}\label{thm-AVR gap corollary}
    There exists a $\dt=\dt(n)>0$ such that the following holds. For a metric soliton  $(X,d,\nu)$ with regular part $(\cR_X,\fg,f_0)$ and a center $x_0$, if there is an $A\ge\dt^{-1}$ such that
$$\frac{|B_{\fg}(x_0,A)\cap\cR_X|}{\omega_nA^n}\ge1-\dt,$$
then the metric soliton is a Gaussian soliton.
\end{Theorem}
To prove this, we need to control the difference between $\dfrac{|B_\fg(x_0,A)\cap\cR_{X}|}{\omega_n A^n}$ and $\cA$ by $o(A)$ so that we can use Theorem \ref{thm-AVR gap for metric soliton}. This could be done by deriving a local $L^{2-\ep}$ bound of $|\Rm|$ for metric solitons, which ensures that the error can be omitted in view of Lemma \ref{lem xvxv}. 

Since metric solitons are more general singularity models of Ricci flows, we expect our gap theorems above to be applied in more scenarios. Indeed, we shall derive an $\varepsilon$-regularity theorem and a gap theorem. Let us, first of all, recall some known $\ep$-regularity theorems in the Ricci flow. The first $\ep$-regularity theorem for the Ricci flow is Perelman's pseudolocality theorem \cite{Per02}.  Hein and Naber \cite{HN14} proved an $\ep$-regularity theorem on the pointed Nash entropy for Ricci flows with  type-I scalar curvature and $\mu$-functional bounded from below. Bamler \cite{Bam20a} generalized this result by removing the assumptions on the scalar curvature and the $\mu$-functional. Cheng and the second author \cite{CZ25} also gave an $\ep$-regularity theorem on Perelman's reduced volume. Li-Qu-Zhu \cite{LQZ26} obtained more $\ep$-regularity theorems for general Ricci solitons. 

Here, we derive an $\ep$-regularity theorem on the local volume ratio.

\begin{Theorem}\label{thm main2}
    There exists a $\dt=\dt(n,R_0)>0$ such that the following holds. Consider any Ricci flow $(M^n,g_t)_{t\in I}$ with bounded curvature on compact time-intervals. Suppose $[t_0-r^2,t_0]\subset I$ for some $r>0$ and $x_0\in M$. If
    \begin{enumerate}[(a)]
        \item $R\le \dfrac{R_0}{t_0-t}$, for all $(x,t)\in  B_{g_{t_0}}(x_0,r)\times[t_0-r^2,t_0]$,
        \item $\dfrac{|B_{g_{t_0}}(x_0,s)|_{g_{t_0}}}{\omega_ns^n}\ge1-\dt$, for all $s\in(0,r]$,
    \end{enumerate}
   then $r_{\Rm}\ge\dt r$. 
\end{Theorem}
\begin{Remark}
\begin{enumerate}[(1)]
    \item If we replace the type-I scalar curvature bound with 
$$R\le \frac{1}{r^2},\,\,\text{for all}\,\, (x,t)\in  B_{g_{t_0}}(x_0,r)\times[t_0-r^2,t_0],$$
then we may obtain a similar result where $\dt$ depends only on $n$;
\item If the type-I scalar curvature  condition is replaced with a type-I Riemannian curvature  condition (i.e. $|\Rm|\le R_0(t_0-t)^{-1}$), then a similar $\varepsilon$-regularity theorem follows immediately from \cite{WW25};

\item One may observe the close relation of the above theorem to the $\varepsilon$-regularity theorem for steady solitons in \cite{LQZ26}; however, the latter cannot be directly derived from applying the former to the canonical flow of a steady soliton, since they do not assume the boundedness of curvature;

\item Our theorem also bears some similarity to Bamler's backward pseudolocality theorem \cite{Bam20b}; however, we do not know how they are exactly related.
\end{enumerate}
\end{Remark}

Next we consider gap theorems for ancient Ricci flows. Yokota \cite{Yo12} proved a gap theorem on the asymptotic reduced volume, under the assumption of a lower bound of the Ricci curvature. The second author \cite{Zhang21} proved a gap theorem on the asymptotic $\cW$-entropy, assuming bounded  geometry. We shall prove the following theorem, which is a limit version of Theorem \ref{thm main2}.
\begin{Theorem}\label{thm main}
There exists a $\dt=\dt(n)>0$ such that the following holds. Consider an ancient Ricci flow $\lt(M^n,g_t\rt)_{t\le0}$ with bounded curvature on compact time-intervals. If 
\begin{enumerate}[(a)]
    \item $R\le\dfrac{R_0}{|t|}$, for all $(x,t)\in M\times(-\infty,0)$,
    \item there exists some $y_0\in M$ such that $\,\,\displaystyle \limsup_{A\to+\infty}\dfrac{|B_{g_0}(y_0,A)|_{g_0}}{\omega_nA^n}\ge1-\dt$,
\end{enumerate}
then the ancient Ricci flow is $\bR^n\times(-\infty,0]$.
\end{Theorem}

\begin{Remark}
    \begin{enumerate}[(1)]
        \item As before, if one replaces the type-I scalar curvature condition with a type-I Riemannian curvature  condition (i.e. $|\Rm|\le R_0|t|^{-1}$), then a similar gap theorem follows from \cite{WW25};
        \item We remind the reader that, Perelman \cite{Per02} alreay proved that an ancient Ricci flow with bounded and nonnegative curvature operator must have zero AVR, unless it is static Euclidean;
        \item Unlike Theorem \ref{thm main2}, the  $\delta$ in the above theorem is independent of the type-I constant $R_0$. 
    \end{enumerate}
\end{Remark}

The proofs of the two theorems above are somehow similar. In both cases, arguing by contradiction, we will properly rescale the flow (or flows) to obtain an $\bF$-limit metric soliton. The type-I scalar curvature bound allows us to use a distance distortion estimate to transfer the volume conditions to the rescaled $-1$ slices, which is furthermore taken to the limit metric soliton, and hence we can use Theorem \ref{thm-AVR gap corollary} to see that the metric soliton we obtained is the Euclidean space; this is enough to get a contradiction.
\\

\emph{Acknowledgement.} The second author is much indebted to Professor Yu Li for many inspiring discussions. Both authors would like to thank Professor Meng Zhu for some helpful personal communications, especially on his manuscript in preparation \cite{LQZ26}.


\section{Preliminaries}

Bamler's $\mathbb{F}$-convergence and $\mathbb{F}$-compactness \cite{Bam20b, Bam23} theorems are central techniques in our work. Because of its extensive length, we shall only introduce the main notions and results applied by us. It is assumed that the readers are generally familiar with the content herein. 

\subsection{$\mathbb{F}$-convergence and $\mathbb{F}$-compactness}

Bamler's $\mathbb{F}$-convergence concerns sequences of metric flow pairs. In \cite[Definition 3.1, Definition 5.1]{Bam23}, a metric flow pair $\lt(\mathcal{X},\lt(\mu_t\rt)_{t\in I}\rt)$ is a couple of an evolving metric space $\cX$ and an evolving probability measure which behaves like a heat kernel.   $\mathbb{F}$-convergence could be viewed as Gromov-Wasserstein convergence for almost every time-slice; see \cite[Definition 5.5, Definition 5.7]{Bam23}. In \cite[Theorem 7.4]{Bam23}, Bamler showed that in the $\bF$-sense, the space of $H$-concentrated metric flow pairs is compact.

In this paper, we only consider the case where the metric flow pairs are actually couples of Ricci flows and  conjugate heat kernels. Let $\lt(M^i,g^i_{t}\rt)_{t\in(-T_i,0]}$ be a sequence of $n$-dimensional Ricci flows with bounded curvature within each compact time interval. For each $i$, let $x_i\in M^i$ be a fixed base point and 
\begin{align*}
    d\nu^i_{x_i,t\,;\, s}:= K^i\lt(x_i,t\,;\,\cdot,s\rt)\,dg^i_s,\quad s\le t
\end{align*}
be the probability measure induced by the conjugate heat kernel based at $(x_i,t)$. Then by \cite{Bam20a}, $ \lt\{ \lt( (M^i,g^i_t )_{t\in(-T_i,0]}, (\nu^i_{x_i,0\,;\,t})_{t\in(-T_i,0]} \rt) \rt\}_{i=1}^\infty$ is a sequence of $H_n$-concentrated metric flow pairs, where
$$H_n=\frac{(n-1)\pi^2}{2}+4.$$
By \cite[Theorem 7.4]{Bam23}, we can find a (not relabeled) subsequence, such that
\begin{align}\label{eq:limiting sequence}
    \left( \lt(M^i,g^i_t\rt)_{t\in(-T_i,0]}, \lt(\nu^i_{x_i,0;t} \rt)_{t\in(-T_i,0]}\right)\xrightarrow[i\to\infty]{\mathbb{F},\ \mathfrak{C}, \ J} \lt(\mathcal{X},\lt(\nu_t\rt)_{t\in(-T,0)}\rt),
\end{align}
where $\mathfrak{C}$ stands for a correspondence (see \cite[Definition 6.1]{Bam23}), $J$ is a prescribed finite set of times such that the convergence is time-wise on $J$. The limit space $\lt(\mathcal{X},\lt(\nu_t\rt)_{t\in(-T,0)}\rt)$ is an $H_n$-concentrated metric flow pair,  $\displaystyle T:=\limsup_{i\to\infty} T_i\in(0,+\infty]$, and $\nu_t$ is a conjugate heat flow satisfying 
$$\operatorname{Var}(\nu_t)\le H_n|t|,$$  
where $\operatorname{Var}$ is the variance.

If the sequence in \eqref{eq:limiting sequence} is additionally noncollapsed, which means that for some $\tau>0$ and $Y<+\infty$,
\begin{align}\label{eq:noncollapsing assumption}
    \mathcal{N}_{x_i,0}(\tau)\ge -Y\quad \text{ for each }\ i,
\end{align}
then Bamler \cite{Bam20b} proved some partial regularity results for the limit flow:

\begin{Theorem}[Bamler's partial regularity result {\cite[Theorem 2.4, Theorem 2.5]{Bam20b}}]\label{thm partial regular}
    If the noncollapsed assumption \eqref{eq:noncollapsing assumption} holds, then the $\bF$-limit $\lt(\mathcal{X},(\nu_t)_{t\in(-T,0)}\rt)$ from \eqref{eq:limiting sequence} satisfies a decomposition $\mathcal{X}=\mathcal{R} \sqcup \mathcal{S}$  such that the following hold.
    \begin{enumerate} [(a)]
        \item The regular part $\mathcal{R}$ is a smooth Ricci flow space-time: $\mathcal{R}$ is locally space-time product, and on each slice $\mathcal{R}_t$, there is a metric $\fg_t$, such that $\fg_t$ satisfies the Ricci flow equation. Furthermore, $d\nu_t=u_td\fg_t$ on $\mathcal{R}_t$, where $u$ is a positive solution to the conjugate heat equation on $\mathcal{R}$;
        \item $\mathcal{S}$ is a set of measure zero for each $t$, and the space-time Minkowski codimension of $\mathcal{S}$ is no smaller than four;
        \item For each $t$, the metric completion of $(\mathcal{R}_t,\fg_t)$ is the metric space $\mathcal{X}_t$;
        \item The convergence is smooth on $\mathcal{R}$ (see Theorem \ref{Thm_smooth_convergence} for more details).
    \end{enumerate}
    Here (and always) $\mathcal{X}_t$, $\mathcal{R}_t$, and $\mathcal{S}_t$ stand for time-slices.
\end{Theorem}

We present a detailed elaboration on the local smooth convergence in Theorem \ref{thm partial regular}(d).

\begin{Theorem}[{\cite[Theorem 9.21]{Bam23} and \cite[Theorem 2.5]{Bam20b}}]\label{Thm_smooth_convergence}
Suppose \eqref{eq:limiting sequence} and \eqref{eq:noncollapsing assumption} both hold, then we can find an increasing sequence $U_1 \subset U_2 \subset \ldots \subset \mathcal{R}$ of open subsets with $\bigcup_{i=1}^\infty U_i = \cR$, open subsets $V_i \subset M^i\times(- T_i,0]$, time-preserving diffeomorphisms $\psi_i : U_i \to V_i$ and a sequence $\ep_i \to 0$ such that the following hold:
\begin{enumerate}[label=(\alph*)]
\item \label{Thm_smooth_convergence_a} We have
\begin{align*}
 \Vert \psi_i^* g^i - \fg \Vert_{C^{[\ep_i^{-1}]} ( U_i)} & \leq \ep_i, \\
  \Vert  u^i \circ \psi_i - u \Vert_{C^{[\ep_i^{-1}]} ( U_i)} &\leq \ep_i, 
\end{align*}
where $d\nu^i_{x_i,0\,;\,t}=u^i(\cdot,t)\,dg^i_t$, $d\nu_{t}=u(\cdot,t)\,d\mathfrak{g}_t$.
\item \label{Thm_smooth_convergence_b} Let $x \in \cR$ and $(x_i,t_i) \in M^i\times(-T_i,0]$.
Then $(x_i,t_i) \to x$ within $\fC$ if and only if $(x_i,t_i) \in V_i \subset M^i\times(-T_i,0]$ for large $i$ and $\psi_i^{-1} (x_i,t_i) \to x$ in $\cR$.
\item \label{Thm_smooth_convergence_c} If the convergence \eqref{eq:limiting sequence} is time-wise at some time $t \in (-T,0]$ for some subsequence, then for any compact subset $K \subset \cR_t$ and for the same subsequence
\[ \sup_{x \in K \cap U_i} d^Z_t \lt(\varphi^i_t (\psi_i(x)), \varphi_t (x) \rt)  \longrightarrow 0, \]
where $\mathfrak{C}:=(Z,\varphi^i_t,\varphi_t)$ is the correspondence.
\item \label{Thm_smooth_convergence_d} Consider a sequence of conjugate heat flows $(\tilde\nu_t^i)_{t < t_0}$ on $M^i\times(-T_i,0]$, $i \in \bN $, for $t_0 \le0$ such that
\[ (\tilde\nu_t^i)_{ t < t_0} \xrightarrow[i \to \infty]{\quad \fC \quad} (\tilde\nu_t)_{ t < t_0}. \]
Write $d\tilde\nu^i_t = \tilde v^i_t \, dg^i_t$  for $i \in \bN$ and $d\tilde\nu_t = \tilde v_t \, d\mathfrak{g}_t$.
Then on $\cR$
\[    \tilde v^i \circ \psi_i  \xrightarrow[i \to \infty]{\quad C^\infty_{\loc} \quad} \tilde v . \]
\end{enumerate}
\end{Theorem}

\subsection{Noncollapsed $\mathbb{F}$-limit metric soliton}

In view of \eqref{eq:limiting sequence}, if a noncollapsed $\bF$-limit space also satisfies that the Nash entropies converge time-wise to a constant function, namely,
\begin{align}\label{eq:entropy converging to const}
    \lim_{i\to\infty}\mathcal{N}_{x_i,0}(\tau)=W\quad \text{ for all }\quad \tau \in(0,T), 
\end{align}
where $W$ is a constant, then we call the limit flow $(\mathcal{X},\nu_t)$ a \emph{noncollapsed $\mathbb{F}$-limit metric soliton}, and $W$ is called the \emph{soliton entropy}. Without loss of generality, we always assume $T>3$. In this paper, our focus is placed on this specific category of limit spaces, which exhibits favorable properties analogous to those of the classical Ricci shrinkers. Let us recall some basic facts about such self-similar metric flows.

\subsubsection{Basic equations and regularity}

\begin{Theorem}[{\cite[Theorem 2.18, Theorem 15.69]{Bam20b}}]\label{thmbase}
    For a noncollapsed $\mathbb{F}$-limit metric soliton $\lt(\mathcal{X},\lt(\nu_t\rt)_{t\in(-T,0)}\rt)$, we write $\tau=-t$, $d\nu_t=(4\pi\tau)^{-\frac{n}{2}}e^{-f}d\fg_t$ on $\mathcal{R}.$ 
    \begin{enumerate}[(a)]
        \item  On the regular part $\mathcal{R}$, $\nabla f$ is complete, and we have
    \begin{equation}\label{eq1}
        \Ric+\nabla^2f-\frac{1}{2\tau}\fg_t=0,\qquad -\tau\left(|\nabla f|^2+R\right)+f\equiv W,
    \end{equation}
where $W$ is the soliton entropy.
\item There is a metric measure space $(X,d, \nu)$, called the model of the metric soliton, such that
$$\mathcal{X}_{<0}=X\times(-T,0),$$
and the following hold for all $t\in(-T,0)$:

\begin{enumerate}[(i)]
    \item $(\mathcal{X}_t,d_t)=(X\times\{t\},|t|^{1/2}d)$.
    \item $\mathcal{R}=\mathcal{R}_X\times(-T,0)$, where $\mathcal{R}_X$, the regular part of $X$, is a smooth manifold.
    \item $(\mathcal{R}_t,\fg_t)=(\mathcal{R}_X\times\{t\},|t|\fg)$, where $\fg$ is a Riemannian metric on $\mathcal{R}_X$.
    \item $\nu_t = \nu$.
\end{enumerate}
Moreover, there is a unique family of probability measures $(\nu'_{x;t})_{x\in X,\, t\ge 0}$ such that the tuple $\lt(X,d,\nu,\lt(\nu'_{x;t}\rt)_{x\in X,\, t\ge 0}\rt)$ is a model for $\lt(\mathcal{X},\lt(\nu_t\rt)_{t\in(-T,0)}\rt)$ in the sense of \cite[Definition 3.57]{Bam23}.

\item The singular part of the model $(X,d,\nu)$ is a null set with codimension no less than four, and $(X,d)$ is the metric completion of $(\mathcal{R}_X,\fg)$.

\item Writing $f_0:=f(\cdot,-1)$, then the tuple $(\mathcal{R}_X,\fg,f_0)$, called the regular part of $(X,d,\nu)$, satisfies the shrinker equations
\begin{align*}
        \Ric+\nabla^2f_0=\frac{1}{2}\fg,\,\,\, |\nabla f_0|^2+R=f_0- W,\,\,\,    \nabla R =2\Ric(\nabla f_0). 
    \end{align*}
\item If $(\mathcal{R}_X,\fg)$ is non-Ricci-flat, then the scalar curvature is positive everywhere on $\mathcal{R}_X$. If $(\mathcal{R}_X,\fg)$ is Ricci-flat, then $(X,d)$ is a metric cone.
\end{enumerate}

\end{Theorem}

\subsubsection{The center and the quadratic growth of the potential function}

By \cite[\S 4.1]{CMZ24}, there exists a point $x_0\in\mathcal{R}_X=\mathcal{R}_{-1}$, called a center of a metric soliton, such that
$$\Var\lt(\nu_t,\delta_{x_0(t)}\rt)\le H_n|t|\quad \text{ for all } t\in(-T,0).$$
Here $x_0(t)$ is the integral curve of $\partial_t-\nabla f_0$ with $x_0(-1)=x_0$ in $\mathcal{R}$, which exists for all $t$. In the following we always use the notation $x_0$ to represent the center.

\begin{Theorem}[{\cite[Lemma 7.12]{FL25}}]\label{thm<f<}
    Let $(X,d,\nu)$ be an n-dimensional non-collapsed $\mathbb{F}$-limit metric soliton,  $(\mathcal{R}_X,\fg,f_0)$ be the regular part, and $x_0\in\mathcal{R}_X$ be a center. For any $\ep>0,$ we have
    \begin{equation*}
        \frac{1}{4+\varepsilon}d^2_{\fg}(x_0,x)-C(n,\varepsilon)\le f_0(x)-W\le\frac{1}{4}\lt(d_{\fg}(x_0,x)+C_1(n)\rt)^2,
    \end{equation*}
    where $W$ is the soliton entropy.
\end{Theorem}

\subsubsection{AVR of metric solitons}\label{Volume sth}

Consider $(\mathcal{R}_X,\fg,f_0)$ with a center $x_0\in\mathcal{R}_X$ as mentioned above. Define $\rho:=2\sqrt{f_0-W}\ge0$, and Theorem \ref{thm<f<} implies
\begin{equation}\label{ineq4}
    \sqrt{\frac{1}{1+\ep}} d_{\fg}(x,x_0)-C(n,\ep)\le\rho(x)\le d_{\fg}(x,x_0)+C(n).
\end{equation}
Similarly to \cite{CZ10} or \cite{Zhang11}, we define
\begin{align}\label{eq defs of AVR}
    D(s)&:=\{x\in \cR_X,\rho(x)\le s\},\nonumber \\
    V(s)&:=\int_{D(s)\cap\mathcal{R}_X}d\fg, \nonumber \\
    \chi(s)&:=\int_{D(s)\cap\mathcal{R}_X}Rd\fg, \nonumber\\
P(s)&:=  \frac{V(s)}{s^n}-4\frac{\chi(s)}{s^{n+2}}.\nonumber\\
\end{align}
Then the following classical equations from \cite{CZ10} are verified in \cite{CMZ24}.
\begin{Lemma}[{\cite[Lemma 8.4]{CMZ24}}]\label{lemVX}
$V(s)$, $\chi(s)$ are absolutely continuous and for almost all $s>0$  we have
\begin{equation}\label{ineq5}
    \frac{n}{2}V(s)-\frac{s}{2}V'(s)=\chi(s)-\frac{2}{s}\chi'(s).
\end{equation}
\begin{equation}\label{ineq6}
    \chi(s)\le\frac{n}{2}V(s).
\end{equation}
\end{Lemma}
We now gather some useful results in \cite{DWZ26} concerning the AVR for metric solitons.
\begin{Theorem}[{\cite[Lemma 3.2, Corollary 3.3, Corollary 3.4]{DWZ26}}]\label{thm AVR basic}
   Let $(X,d,\nu)$ be a noncollapsed $\mathbb{F}$-limit metric soliton, $(\mathcal{R}_X,\fg,f_0)$ its regular part, and $x_0$ be a center. 
   \begin{enumerate}[(a)]
       \item  The the $\AVR$  $\displaystyle \cA=\lim_{A\to+\infty}\frac{|B_{\fg}(x_0,A)\cap\cR_{X}|}{\omega_n A^n}$ exists and 
    $$\omega_n\cA=\lim_{s\to+\infty}P(s) =\lim_{s\to+\infty}\frac{V(s)}{s^n}.$$ 
    
    \item  $\cA>0$ if and only if there exists $s_0>0$, such that $\displaystyle \int_{s_0}^{+\infty}\frac{\chi(t)}{tV(t)}dt<+\infty$.\\
    \item  We have $ V(s)\le C(n)e^W s^n$ for all $s\ge\underline{A}(n)$, where $W\le 0$ is the soliton entropy. In particular, $\cA$ satisfies a bound 
    $$\cA\in\left[0,C(n)e^W\right].$$
   \end{enumerate}
\end{Theorem}

We also need an extra preparations.

\begin{Lemma}\label{lem xvxv}
    If $\cA>0$, then 
    $$\lim_{s\to+\infty}\frac{\chi(s)}{V(s)}=\lim_{s\to+\infty}\frac{\chi'(s)}{V(s)s}=0.$$
\end{Lemma}
\begin{proof}
    We differentiate $P$ and use \eqref{ineq5} to get
    \begin{equation*}
        P'(s)=\frac{2\chi(s)}{s^{n+1}}\left(\frac{2n+4}{s^2}-1\right).
    \end{equation*}
    Assume $s\ge\sqrt{2n+4}$, and by Theorem \ref{thm AVR basic} (a), we have
    \begin{align*}
        \cA\omega_n-P(s)&=\int_s^{\infty}\frac{2\chi(r)}{r^{n+1}}\left(\frac{2n+4}{r^2}-1\right)dr\\
        &\le \chi(s)\int_s^{\infty}\frac{2}{r^{n+1}}\left(\frac{2n+4}{r^2}-1\right)dr\\
        &=\chi(s)\left(\frac{4}{s^{n+2}}-\frac{2}{ns^n}\right).
    \end{align*}
    So we have
    \begin{equation*}
        \frac{\chi(s)}{s^n}\le\frac{n}{2}\left(\frac{V(s)}{s^n}-\cA\omega_n\right)\xrightarrow{\ s\to+\infty\ }0.
    \end{equation*}
    Then $\displaystyle\lim_{s\to+\infty}\frac{\chi(s)}{V(s)}=0$ follows from using Theorem \ref{thm AVR basic}(a) again.
    
    Now consider $\displaystyle \lim_{s\to+\infty}\frac{\chi'(s)}{V(s)s}$. Then by \eqref{ineq6}, for $s$ large enough, there exists an $s_0\in[s,s+1]$ such that
    \begin{equation*}
        \chi'(s_0)=\chi(s+1)-\chi(s)\le\chi(s+1)\le\frac{n}{2}V(s+1).
    \end{equation*}
    So for all $s$ large enough, we have that
    \begin{equation*}
        \chi'(s)\le\frac{n}{2}V(s+1)\le C(n)(s+1)^n,
    \end{equation*}
    so  $\displaystyle \lim_{s\to+\infty}\frac{\chi'(s)}{V(s)s}=0$ follows immediately.
\end{proof}


\section{An AVR gap theorem for metric solitons}
In this section, we prove Theorem \ref{thm-AVR gap for metric soliton} and Theorem \ref{thm-AVR gap corollary}. The idea of our proof is borrowed from \cite{WW25}.

 Let $(M^i,g^i_t,x_i)_{t\in[-T_i,0]}$ be a sequence of $n$-dimensional Ricci flows, where $x_i\in M^i$ and each Ricci flow in this sequence has bounded curvature within each compact time-interval. Suppose that the sequence is noncollapsed, namely, there is a $\tau>0$ and $Y<+\infty$, such that
\begin{align}\label{ineq N>}
    \cN_{x_i,0}(\tau)\ge-Y,\,\, \text{for all } i.
\end{align}
Then there exists an $\bF$-limit space $\lt(\cX,(\nu_t)_{t\in(-T,0)}\rt)$ and a correspondence $\fC$, such that, after passing to a subsequence,
\begin{align}\label{def-limiting-condition}
    \left( \lt(M^i,g^i_t\rt)_{t\in \lt(-T_i,0 \rt]}, \lt(\nu^i_{x_i,0;t} \rt)_{t\in(-T_i,0]}\right)\xrightarrow[i\to\infty]{\mathbb{F},\mathfrak{C}}\lt(\mathcal{X},\lt(\nu_t\rt)_{t\in(-T,0)}\rt).
\end{align}
Recall that we have assumed $T>3$ without loss of generality.

We furthermore suppose that
\begin{align}\label{eq limN}
    \lim_{i\to+\infty}\cN_{x_i,0}(\tau)=W \quad \text{ for all }\quad \tau \in(0,T).
\end{align}
Then $\lt(\cX,\lt(\nu_t\rt)_{t\in(-T,0)}\rt)$ is a metric soliton. We may also assume that the convergence \eqref{def-limiting-condition} is time-wise for all $t\in(-T,0)$ in view of \cite[Lemma 4.1]{DWZ26}. Let $(X,d,\nu)$ be the model of the metric soliton with regular part $(\cR_{X},\fg,f_0)$. Fix a center $x_0\in\cR_X$.

\subsection{An integration-by-parts argument}

We will show that on the regular part $(\cR_X,\fg,f_0)$, the integration by parts of the potential function $f_0$ is valid, which shall be used later.
\begin{Lemma}\label{lem IBP of f_0}
For all $\tau>0$, we have
    $$\int_{\cR_{X}}\Delta f_0e^{-\frac{f_0}{\tau}}d\fg=\frac{1}{\tau}\int_{\cR_{X}}|\nabla f_0|^2e^{-\frac{f_0}{\tau}}d\fg$$
\end{Lemma}
\begin{proof}
    Let $\phi\in C_0^{\infty}([0,+\infty))$ with $\supp \phi\Subset[0,1]$, and $\phi\equiv1$ on $[0,0.9]$ and $|\phi'|\le 100$. Define 
    \begin{align*}
        \phi_A:=\phi\left(\frac{f_0-W}{\frac{1}{4}(A+C_1(n))^2}\right),\, A\gg 1,
    \end{align*}
   where $C_1(n)$ is the constant in the upper bound of Theorem \ref{thm<f<}. Let  $\eta_{r}$ be the cut-off function constructed in \cite[Proposition 1.6]{CMZ24}. Clearly $\phi_A\eta_r\in C_0^{\infty}(\cR_{X})$ and $\supp(\phi_A\eta_r)\Subset B_{\fg}(x_0,A)\cap\mathcal{R}_X$. By Theorem \ref{thm<f<} again, we adjust the $\ep$ therein and set $A\ge\underline{A}(n)$, so that
   \begin{align}\label{ineq phiA}
       \{\nabla\phi_A\neq0\}\subset B_{\fg}(x_0,1.1A)\setminus B_{\fg}(x_0,0.9A).
   \end{align}
   
    With the cut-off function, we have the validity of integration by parts:
    \begin{align*}
       \int_{\cR_{X}}\Delta \left(\eta_r\phi_A\frac{f_0}{\tau}\right)e^{-\frac{f_0}{\tau}}d\fg=-\int_{\cR_{X}}  \nabla\left(\eta_r\phi_A \frac{f_0}{\tau}\right)\cdot \nabla e^{-\frac{f_0}{\tau}} d\fg.
    \end{align*}
    A routine computation yields
    \begin{align}\label{ineq IBP}
        \left|\int_{\cR_{X}}\eta_r\phi_A\left(\Delta f_0-\frac{|\nabla f_0|^2}{\tau}\right)e^{-\frac{f_0}{\tau}}d\fg\right| &=\left|-\int_{\cR_{X}}e^{-\frac{f_0}{\tau}}\nabla \lt(\eta_r\phi_A \rt)\cdot\nabla f_0d\fg\right|  \nonumber \\
        &\le\int_{\cR_{X}}e^{-\frac{f_0}{\tau}}|\nabla f_0|\lt( \phi_A|\nabla\eta_r|+\eta_r|\nabla\phi_A| \rt)d\fg \nonumber \\
        &\le\int_{\cR_{X}\cap B_{\fg}(x_0,A)\cap\{0<\eta_r<1\}}e^{-\frac{f_0}{\tau}} |\nabla f_0||\nabla\eta_r| d\fg \nonumber \\
                 &\quad+  \int_{\cR_{X}\cap\lt(B_{\fg}(x_0,1.1A)\setminus B_{\fg}(x_0,0.9A)\rt)} e^{-\frac{f_0}{\tau}}|\nabla f_0||\nabla\phi_A|d\fg, \nonumber \\     
    \end{align}
    where we used \eqref{ineq phiA} in the second inequality. Now let us estimate the two terms of \eqref{ineq IBP}, and we will use Theorem \ref{thm<f<} and Theorem \ref{thmbase}(d) frequently in the rest of the proof. For the first term, we may apply \cite[Proposition 1.6(2)(4)]{CMZ24} to get
    \begin{align}\label{ineq IBP part1}
        \int_{\cR_{X}\cap B_{\fg}(x_0,A)\cap\{0<\eta_r<1\}}e^{-\frac{f_0}{\tau}} |\nabla\eta_r||\nabla f_0| d\fg\le C(n,\tau,A,W,\sigma)r^{3-\sigma},
    \end{align}
where $\sigma\in(0,1)$ is a positive constant (which can be conveniently fixed, say, $\sigma=0.5$). For the second term, by the definition of $\phi_A$, we have for $A\ge\underline{A}(n)$,
 $$|\nabla\phi_A||\nabla f_0|\le C(n)A^{-2}|\nabla f_0|^2\le C(n).$$ 
Hence
 \begin{align}\label{ineq IBP part2}
     \int_{\cR_{X}\cap\lt(B_{\fg}(x_0,1.1A)\setminus B_{\fg}(x_0,0.9A)\rt)} e^{-\frac{f_0}{\tau}}|\nabla\phi_A||\nabla f_0|d\fg&\le C(n,\tau,W,\ep)e^{-\frac{0.81A^2}{(4+\ep)\tau}}\int_{\cR_{X}\cap B_{\fg}(x_0,1.1A)} d\fg \nonumber  \\
     &\le C(n,\tau,W,\ep)e^{-\frac{0.81A^2}{(4+\ep)\tau}}A^n,
 \end{align}
 where we used Theorem \ref{thm AVR basic}(c) in the second inequality.
 Combining \eqref{ineq IBP}, \eqref{ineq IBP part1}, and \eqref{ineq IBP part2} yields
 \begin{align*}
     \int_{\cR_{X}}\eta_r\phi_A\left(\Delta f_0-\frac{|\nabla f_0|^2}{\tau}\right)e^{-\frac{f_0}{\tau}}d\fg\le C(n,\tau,W,\ep)e^{-\frac{0.81A^2}{(4+\ep)\tau}}A^n+C(n,\tau,A,W,\sigma)r^{3-\sigma}.
 \end{align*}
If we let $r\to0$ first and then let $A\to+\infty$, then the right side tends to zero, and the proof is completed. 
\end{proof}

\subsection{$h$-integral and its monotonicity}
Define
\begin{equation}
    h(\tau):=\int_{\cR_{X}}(4\pi\tau)^{-\frac{n}{2}}e^{-\frac{f_0-W}{\tau}}d\fg,\,\,\tau>0.
\end{equation}
From the almost quadratic growth of $f_0$, we see that the integral in $h$ is well-defined. To prove Theorem \ref{thm-AVR gap for metric soliton}, we need the following proposition which establishes a connection between the soliton entropy and the AVR. By Theorem \ref{thm AVR basic}, the AVR of the metric soliton exists and we denote it by $\cA$.
\begin{Proposition}\label{prop h monotone}
$h(\tau)$ is increasing for $\tau<1$ and decreasing for $\tau>1$. If $\cA>0$, we have 
\begin{equation}\label{eq lim h=A}
    \lim_{\tau\to+\infty}h(\tau)=\cA.
\end{equation}
\end{Proposition}
\begin{proof}
    The proof is similar to \cite[Proposition 2.6]{WW25}. First we compute
    \begin{equation}\label{eq h'}
        \frac{d}{d\tau}h(\tau)=(4\pi)^{-n/2}\tau^{-\frac{n}{2}-1}\int_{\cR_{X}}\left(\frac{f_0-W}{\tau}-\frac{n}{2}\right)e^{-\frac{f_0-W}{\tau}}d\fg.
    \end{equation}
    From Theorem \ref{thmbase}(d), we get
    \begin{equation}\label{eq fff}
        \Delta f_0+R=\frac{n}{2},\,\, R+|\nabla f_0|^2=f_0-W,\,\, \Delta f_0-|\nabla f_0|^2+f_0=\frac{n}{2}+W.
    \end{equation}
  By Lemma \ref{lem IBP of f_0}, we have that
    \begin{equation}\label{eq IBP}
        \int_{\cR_{X}}\Delta f_0  e^{-\frac{f_0-W}{\tau}}d\fg=\int_{\cR_{X}} \frac{|\nabla f_0|^2}{\tau} e^{-\frac{f_0-W}{\tau}}d\fg.
    \end{equation}
    Substituting \eqref{eq fff} and \eqref{eq IBP} into \eqref{eq h'} yields
    \begin{align*}
        \frac{d}{d\tau}h(\tau)&=(4\pi)^{-n/2}\tau^{-\frac{n}{2}-1}\int_{\cR_{X}}\frac{1-\tau}{\tau}\lt(f_0-W-|\nabla f_0|^2\rt)e^{-\frac{f_0-W}{\tau}}d\fg\\ 
        &=(4\pi)^{-n/2}\tau^{-\frac{n}{2}-2}\int_{\cR_{X}}(1-\tau)Re^{-\frac{f_0-W}{\tau}}d\fg.
    \end{align*}
     Thus, we obtain the monotonicity of $h(\tau)$ since $R\ge0$.

    Next we prove \eqref{eq lim h=A}. We divide $h(\tau)$ into two parts:
    \begin{equation*}
    h(\tau)=\int_{D(s)}(4\pi\tau)^{-\frac{n}{2}}e^{-\frac{f_0-W}{\tau}}d\fg+\int_{\cR_{X}\setminus D(s)}(4\pi\tau)^{-\frac{n}{2}}e^{-\frac{f_0-W}{\tau}}d\fg,
\end{equation*}
where $D(s)$ and $V(s)$ below are defined in \eqref{eq defs of AVR}. For the first part, we simply have 
\begin{equation}\label{ineq h1part}
    \int_{D(s)}(4\pi\tau)^{-\frac{n}{2}}e^{-\frac{f_0-W}{\tau}}d\fg\le(4\pi\tau)^{-\frac{n}{2}}V(s).
\end{equation}
Now we consider the second part. Since $\cA>0$, we apply Theorem \ref{thm AVR basic} to obtain that 
\begin{equation*}
    \lim_{s\to+\infty}\frac{V(s)}{s^n}=\cA\omega_n.
\end{equation*}
Multiplying both sides of \eqref{ineq5} by $V(s)^{-1}$ and using Lemma \ref{lem xvxv}, we have 
\begin{equation*}
    \lim_{s\to+\infty}\frac{sV'(s)}{V(s)}=n.
\end{equation*}
Then for any $\ep>0$, then there exists an $r=r(\ep)$ large enough such that for all $s\ge r$, we have
\begin{align*}
    \frac{sV'(s)}{V(s)}&\le n+\ep\\
    &=(n+\ep)\left(\frac{V(r)}{r^n}\frac{s^n}{V(s)}\right)^{-1}\left(\frac{V(r)}{r^n}\frac{s^n}{V(s)}\right)\\
    &\le(n+\ep)(1+\ep)\frac{V(r)}{r^n}\frac{s^n}{V(s)}\\
    &\le(n+\ep)\frac{V(r)}{r^n}\frac{s^n}{V(s)}.
\end{align*}
This means  
\begin{equation}\label{ineq v'<V}
    V'(s)\le (n+\ep)\frac{V(r)}{r^n}s^{n-1},\,\,s\ge r.
\end{equation}
Applying the regularity result \cite[Corollary 1.3]{Kot13}, we get in the geodesic normal coordinates, both $\fg$ and $\nabla f_0$ are real analytic. Therefore, both $f_0=W+R+|\nabla f_0|^2$ and $R$ are real analytic functions on $\cR_X$. Next we use \cite{M20} on the analytic function $|\nabla f_0|^2=f_0-R-W$ and we obtain that the singular set 
$$\cC=\{x\in\cR_X: |\nabla\rho|(x)=0\}$$ 
has zero measure, which means $|\nabla\rho|^{-1}$ is a nonnegative measurable function on $\cR_X\setminus\cC$. Furthermore, by the Morse-Sard Theorem \cite{SS72}, the singular value set $\rho(\cC)$ is at most countable.

We now apply the co-area formula to $\rho$ on $\cR_X\setminus\cC$ and get
\begin{eqnarray*}
V(s)&=& |D(s)\cap \mathcal{R}_X\cap \mathcal{C}|+|D(s)\cap \mathcal{R}_X\setminus\mathcal{C}|= |D(s)\cap \mathcal{R}_X\setminus\mathcal{C}|\\
&=&\int_{D(s)\cap \mathcal{R}_X\setminus\mathcal{C}}\,dg =\int_0^s\int_{\{x\in \mathcal{R}_X\setminus\mathcal{C}:\, \rho(x)=t\}}\frac{1}{|\nabla \rho|}\,dA\, dt.
\end{eqnarray*}
This implies
\begin{align*}
    V'(s)=\int_{\{x\in \mathcal{R}_X\setminus\mathcal{C}:\, \rho(x)=s\}}\frac{1}{|\nabla \rho|}\,dA.
\end{align*}
We apply the co-area formula again on $\displaystyle\int_{\cR_{X}\setminus D(s)}(4\pi\tau)^{-\frac{n}{2}}e^{-\frac{f_0-W}{\tau}}d\fg$, and combining the above equalities with \eqref{ineq v'<V} yields 
\begin{align}\label{ineq h2part}
    \int_{\cR_{X}\setminus D(s)}(4\pi\tau)^{-\frac{n}{2}}e^{-\frac{f_0-W}{\tau}}d\fg&=\int_s^{+\infty} \int_{\{x\in \mathcal{R}_X\setminus\mathcal{C}:\, \rho(x)=t\}}(4\pi\tau)^{-\frac{n}{2}}\frac{1}{|\nabla \rho|}e^{-\frac{f_0-W}{\tau}}\,dA\,dt \nonumber \\ 
    &=\int_s^{+\infty}(4\pi\tau)^{-\frac{n}{2}}e^{-\frac{t^2}{4\tau}}V'(t)dt \nonumber \\ 
    &\le(n+\ep)\frac{V(r)}{r^n}(4\pi\tau)^{-\frac{n}{2}}\int_s^{+\infty}e^{-\frac{t^2}{4\tau}}t^{n-1}dt \nonumber\\
    &\overset{\eta=\frac{t^2}{4\tau}}{=}\frac{n+\ep}{2}\frac{V(r)}{r^n}\pi^{-\frac{n}{2}}\int_{\frac{s^2}{4\tau}}^{+\infty}e^{-\eta}{\eta}^{\frac{n-2}{2}}d\eta \nonumber\\
    &\le\frac{n+\ep}{n\omega_n}\frac{V(r)}{r^n},
\end{align}
where we used 
\begin{equation*}
    1=\int_{\bR^n}\pi^{-\frac{n}{2}}e^{-|x|^2}dx=\frac{n\omega_n}{2}\pi^{-\frac{n}{2}}\int_0^{+\infty}e^{-t}t^{\frac{n-2}{2}}dt
\end{equation*}
to get the last inequality. Now the upper bound of $h$ follows from \eqref{ineq h1part} and \eqref{ineq h2part}, namely,
\begin{equation*}
    h(\tau)\le(4\pi)^{-\frac{n}{2}}\frac{V(r)}{\tau^{\frac{n}{2}}}+\frac{n+\ep}{n\omega_n}\frac{V(r)}{r^n}.
\end{equation*}
Now set $\tau\to+\infty$ first and then $r\to+\infty$, which leads to
\begin{equation*}
    \lim_{\tau\to+\infty}h(\tau)\le\cA.
\end{equation*}
To see that $\displaystyle\lim_{\tau\to+\infty}h(\tau)\ge\cA$, we only need to replace \eqref{ineq v'<V} by 
\begin{equation*}
    V'(s)\ge (n-\ep)\frac{V(r)}{r^n}s^{n-1}
\end{equation*}
and repeat the above process.
\end{proof}

\begin{proof}[Proof of Theorem \ref{thm-AVR gap for metric soliton}]
It suffices to show that there exists a $\delta=\delta(n)$ such that if $\cA\ge1-\delta$, then the metric soliton is Gaussian. Since the AVR is positive, we have by Proposition \ref{prop h monotone} that
\begin{equation*}
    \cA\le h(1)=e^W.
\end{equation*}
Set $\dt=\dt(n)=1-e^{-\ep}$, where $\ep=\ep(n)$ is the soliton entropy gap constant from \cite[Proposition 6.1]{CMZ24}. Then $\cA\ge1-\dt=e^{-\ep}$ implies $W\ge-\ep$, which furthermore implies that the metric soliton in question is the Gaussian soliton. This  completes the proof.
\end{proof}

\subsection{An asymptotic volume gap theorem for metric solitons}
We now prove Theorem \ref{thm-AVR gap corollary} by using a similar method as in \cite[Corollary 2.9]{WW25}. We still consider a metric soliton $(X,d,\nu)$ with regular part $(\mathcal{R}_X,\fg,f_0)$ as introduced at the beginning of the section. The following lemma is needed to control the local $L^{2-\ep}$ bound of $|\Rm|$ for metric solitons.

\begin{Lemma}\label{lm Rm lp}
    There exists some $\underline{A}=\underline{A}(n)$, such that for any $A\ge\underline{A}$ and any $\ep>0$, we have 
    \begin{equation*}
        \int_{B_{\fg}(x_0,A)\cap\cR_X}|\Rm|^{2-\ep}d\fg\le C(n,W,\ep)A^{n-2+2\ep}.
    \end{equation*}
\end{Lemma}
\begin{proof}
     We first shift our focus from the model space to the metric flow.  By \cite[Theorem 3.1]{CMZ24}, we may, without loss of generality, assume the metric soliton in \eqref{def-limiting-condition} and \eqref{eq limN} is defined over $(-\infty,0)$ and $T_i\nearrow+\infty$. Let $x_0(t)$ be the integral curve of $\partial_{\ft}-\nabla f$ with a center of the metric soliton $x_0(-1)=x_0\in\cR_{-1}$. Define 
    \begin{equation*}
        P^*_{\nu}(A,-T'):=\left\{x\in\cX_{[-T',0)},d_{W_1}^{-T'}(\nu_{x;-T'}\,:\,\nu_{-T'})<A\right\}.
    \end{equation*}
    By \cite[Lemma 6.4]{CMZ24}, if $A\ge\underline{A}=(\sqrt{2}-1)^{-1}\sqrt{H_n}$, we have 
    \begin{equation*}
        B_{\fg_t}\lt( x_0(t), A\sqrt{|t|}\rt)\subset P_{\nu}^*\lt(2A,-2\rt),\,\, \text{for all}\,\,t\in[-2,0).
    \end{equation*}
     We shall apply \cite[Theorem 2.31]{Bam20b}. Note that the parameters $Y$ in the statement thereof are absorbed into the parameter $W$ in our case due to \eqref{eq limN}. Hence for any $\ep>0$,
    \begin{equation*}
       \int_{-2}^{-1/2}\int_{B_{\fg_t}\lt(x_0(t),A\sqrt{|t|}\rt)\cap\cR_t}|\Rm|^{2-\ep}d\fg_tdt \le  \int_{-2}^0\int_{P^*_{\nu}\lt(2A,-2\rt)\cap\cR_t}|\Rm|^{2-\ep}d\fg    \le C(n,W,\ep)A^{n-2+2\ep}.
    \end{equation*}
    Then there exists a $t'\in[-2,-1/2]$, such that 
    \begin{equation*}
        \int_{B_{\fg_{t'}}\lt(x_0(t'),A\sqrt{|t'|}\rt)\cap\cR_{t'}}|\Rm|^{2-\ep}d\fg_{t'}\le 2C(n,W,\ep)A^{n-2+2\ep}.
    \end{equation*}
    The geometry of $\cR_{t'}\subset\cX_{t'}$ and $\cR_X\subset \cX$ differs only by a scaling with some bounded factor $t'\in[-2,-1/2]$, where $B_{\fg_{t'}}\lt(x_0(t'),A\sqrt{|t'|}\rt)\subset\cX_{t'}$ corresponds to $B_{\fg}(x_0,A)\subset\cX$. The two time-slices could be identified via an isometry in view of Theorem \ref{thmbase} (b) and \cite[Proposition 3.3]{CMZ24}. The proof is done.

\end{proof}

\begin{proof}[Proof of Theorem \ref{thm-AVR gap corollary}]
   Set $\dt_0>0$ and $A_0<+\infty$ to be determined. Assume the metric soliton in question satisfies
   \begin{align}\label{ineq contradict}
       \frac{|B_\fg(x_0,A_0)\cap\cR_{X}|}{\omega_n A_0^n}\ge1-\dt_0.
   \end{align}
We need to control the error term of the integral of $R$.  By Lemma \ref{lm Rm lp} and \cite[Corollary 1.2]{CMZ24}, we have 
   \begin{align*}
       \int_{B_{\fg}(x_0,A)\cap\cR_{-1}}Rd\fg &\le \left(\int_{B_{\fg}(x_0,A)\cap\cR_{-1}}R^{2-\ep}d\fg\right)^{\frac{1}{2-\ep}}|B_{\fg}(x_0,A)\cap\cR_{-1}|^{\frac{1-\varepsilon}{2-\varepsilon}}\\
       &\le C_2(n,\ep)A^{n+1-\frac{4-3\ep}{2-\ep}}
   \end{align*}
   holds for all $A\ge\underline{A}$, where $\underline{A}=\underline{A}(n)$ is the constant defined in Lemma \ref{lm Rm lp}. Set $\ep=2/3$ and note that \eqref{ineq4} gives an equivalence between $D(A)$ and $B_\fg(x_0,A)$, so this means for $A\ge \underline{A}$,
   \begin{align}\label{ineq chi}
       \frac{\chi(A)}{A^{n+1}}\le C_3(n)A^{-\frac{3}{2}}.
   \end{align}
Then by the proof of Lemma \ref{lem xvxv}, for $A\ge\underline{A}$,
   \begin{align*}
       \cA\omega_n-P(A)&=\int_{A}^{+\infty}2\frac{\chi(s)}{s^{n+1}}\left(\frac{2n+4}{s^2}-1\right)ds\ge-C_3\int_{A}^{+\infty}s^{-\frac{3}{2}}ds=-\frac{2C_3}{\sqrt{A}}.
   \end{align*}
The definition of $P$ and \eqref{ineq contradict} imply that
   \begin{align*}
       \cA&\ge\frac{V(A)}{\omega_n A^n}-4\frac{\chi(A)}{\omega_nA^{n+2}}-2\frac{C_3}{\omega_n\sqrt{A}}\\
       &\ge \frac{|B_{\fg}\lt(x_0,A-C(n)\rt)\cap\cR_{X}|}{\omega_n\lt(A-C(n)\rt)^n} \left(\frac{A-C(n)}{A}\right)^n -4\frac{\chi(A)}{\omega_nA^{n+2}}-2\frac{C_3}{\omega_n\sqrt{A}}\\
       &\ge (1-\dt_0) \left(\frac{A_0}{A_0+C(n)}\right)^n -4\frac{\chi \left( A_0+C(n) \right)}{\omega_n \lt( A_0+C(n) \rt)^{n+2}}-2\frac{C_3}{\omega_n\sqrt{A_0+C(n)}},
   \end{align*}
where we let $A_0$ in \eqref{ineq contradict} be $A-C(n)$ to get the last inequality and $C(n)$ is the upper bound in \eqref{ineq4}. By \eqref{ineq chi}, if we set $\dt_0=\dt/2$ and $A_0\ge\underline{A_0}(\dt)$,  where $\dt=\dt(n)$ is the gap constant in Theorem \ref{thm-AVR gap for metric soliton} we can set
   \begin{align*}
       \cA\ge (1-\dt_0) \left(\frac{A_0}{A_0+C(n)}\right)^n -4\frac{\chi \left( A_0+C(n) \right)}{\omega_n \lt( A_0+C(n) \rt)^{n+2}}-2\frac{C_3}{\omega_n\sqrt{A_0+C(n)}}\ge1-\dt,
   \end{align*}
   and this theorem follows from Theorem \ref{thm-AVR gap for metric soliton}.
\end{proof}

\section{Applications to Ricci flows}
In this section, we focus on Ricci flows, and we will prove Theorem \ref{thm main2} and Theorem \ref{thm main}.

\subsection{Volume estimate under $\bF$-convergence}

We shall derive the ``volume continuity'' with respect to $\bF$-convergence.  All the discussions in this subsection could be viewed as a slight generalization of \cite[Lemma 4.4, Lemma 4.5]{DWZ26}.

Let us recall the settings at the beginning of the previous section. Assume that there is a metric soliton $\lt(\cX,\lt(\nu_t\rt)_{t\in(-T,0)}\rt)$ obtained in \eqref{ineq N>}, \eqref{def-limiting-condition}, and \eqref{eq limN}. Recall that we have assumed, without loss of generality, that $T>3$. The metric soliton is modeled on $(X,d,\nu)$ with regular part $(\cR_X,\fg,f_0)$; these are simply the $t=-1$ slices of $\mathcal{X}$ and $\mathcal{R}$, respectively. Let $x_0\in \cR_X$ be a fixed center of the metric soliton. 

Let $(z_i,-1)$ be an $H_n$-center of $(x_i,0)$ with respect to the flow $g^i_t$. Our goal is to approximate $\lt|B_{\fg}(x_0,A)\rt|$ by $\lt|B_{g^i_{-1}}(z_i,A)\rt|_{g^i_{-1}}$ for $i$ large enough and $A\gg1$. To this end, we decompose $B_{g^i_{-1}}(z_i,A)$ into two parts:
    \begin{align*}
        \cR^{(i)}(r,A):=B_{g^i_{-1}}(z_i,A)\cap\{r_{\Rm}\ge r\},\quad \cS^{(i)}(r,A):=B_{g^i_{-1}}(z_i,A)\cap\{r_{\Rm}<r\},\quad r>0,\ \ A<+\infty.
    \end{align*}
    Here $r_{\Rm}$ stands for the parabolic curvature radius, namely, the largest $r$ such that the curvature is smaller than $r^{-2}$ on $P(x,t;r):=B_{g_t}(x,r)\times[t-r^2,t+r^2]\cap I$.  Obviously, 
    \begin{align*}
        B_{g^i_{-1}}(z_i,A) =\cR^{(i)}(r,A) \sqcup \cS^{(i)}(r,A).
    \end{align*}
    In this subsection, we shall estimate  the volumes of $\cS^{(i)}(r,A)$ and  $\cR^{(i)}(r,A)$, respectively. 
    
    The singular part $\cS^{(i)}(r,A)$ is estimated by Bamler's quantitative stratification theorem.
\begin{Lemma}\label{Lm: volume of the singular part}
        If 
        \begin{align*}
            \quad A\ge 1,\quad r\le\overline{r}(n,A,Y),
        \end{align*}
        then there is an $\alpha=\alpha(n,Y)>0$, such that
        \begin{align*}
            \left|\cS^{(i)}(r,A)\right|_{g^i_{-1}}\le C(n,A,Y)r^{\alpha}\quad \text{ for all }\quad i.
        \end{align*}
        Here $Y$ is the entropy bound in \eqref{ineq N>}.
    \end{Lemma}

\begin{proof}
  By \cite[Corollary 17.2]{Bam20b}, if $\ep= \ep(n,Y)\in(0,1)$ and $r<\varepsilon$, then for any $r'\ge r/\varepsilon$, the points in $\cS^{(i)}(r,A)$ are neither weakly $(n-1,\ep,r')$-split, nor $(\ep,r')$-static and weakly $(n-3,\ep,r')$-split at the same time. Thus, 
  \begin{align*}
      \cS^{(i)}(r,A) \subset \tilde{\mathcal{S}}_{r/\varepsilon,1}^{\varepsilon,n-2}(g^i),
  \end{align*}
  where $\tilde{\mathcal{S}}_{r/\varepsilon,1}^{\varepsilon,n-2}(g^i)$ are the quantitative strata with respect to the flow $g^i_t$ in \cite[Definition 11.1]{Bam20b}. By \cite[Theorem 9.11]{Bam20a}, we can find a set of points $\{y_j^{(i)}\}_{j=1}^{N^{(i)}}\subset \cS^{(i)}(r,A)$ such that
  \begin{equation*}
      \cS^{(i)}(r,A)\times\{-1\}\subset \bigcup_{j=1}^{N^{(i)}} P^*\left(y_j^{(i)},-1;1\right),\,\,N^{(i)}\le \overline{N}(n,A).
  \end{equation*}
See also \cite{Bam20a} the definition of $P^*$-parabolic neighborhoods. This can be done because
\begin{align*}
    \cS^{(i)}(r,A)\times\{-1\} \subset B_{g^i_{-1}}(z_i,A)\times\{-1\}\subset P^*\left(z_i,-1; A,1,-1\right).
\end{align*}
Since $\varepsilon<1$, we have
  \begin{align}\label{strata2}
      \cS^{(i)}(r,A)\times\{-1\}\subset \bigcup_{j=1}^{N^{(i)}} \left(\tilde{\mathcal{S}}_{r/\varepsilon,1}^{\varepsilon,n-2}(g^i)\cap P^*\left(y_j^{(i)},-1;1/\varepsilon\right)\right).
  \end{align}
  
  Now, applying  \cite[Proposition 11.2]{Bam20b} to $\tilde{\mathcal{S}}_{r/\varepsilon,1}^{\varepsilon,n-2}(g^i)\cap P^*\left(y_j^{(i)},-1;1/\varepsilon\right)$ with $\varepsilon\to\varepsilon$, $r\to 1/\varepsilon$, $\sigma \to r$, we can find $\lt\{y^{(i)}_{j,k}\rt\}_{k=1}^{N^{(i)}_j}$, such that
\begin{align}\label{strata3}
    \tilde{\mathcal{S}}_{r/\varepsilon,1}^{\varepsilon,n-2}(g^i)\cap P^*\left(y_j^{(i)},-1;1/\varepsilon\right) \subset \bigcup_{k=1}^{N^{(i)}_j} P^*\left(y_{j,k}^{(i)},-1;r/\varepsilon\right),\quad N^{(i)}_j \le C(n,Y,\varepsilon) r^{-n+2-\varepsilon}.
\end{align}
Since, by \cite[Theorem 9.8]{Bam20a}, 
\begin{align}\label{strata4}
    \left|P^*\left(y_{j,k}^{(i)},-1;r/\varepsilon\right)\bigcap M\times\{-1\}\right|_{g^i_{-1}}\le C(n,Y)r^n\varepsilon^{-n},\quad k= 1,2,\hdots,N^{(i)}_j,\ \ j=1,2,\hdots, N^{(i)},
\end{align}
we have, combining assertion(c) with \eqref{strata2}-\eqref{strata4},
\begin{align*}
    \left|\cS^{(i)}(r,A)\right|_{g^i_{-1}}\le C(n,Y,A,\varepsilon)r^{2-\varepsilon}.
\end{align*}
Since $\varepsilon=\varepsilon(n,Y)$, we obtain the lemma by letting $\alpha=2-\varepsilon>0$.

\end{proof}

\begin{Lemma}\label{Lm:smooth volume convergence}
    If $A\ge\underline{A}(n)$, $r\in(0,1)$, then passing to a subsequence,  we have
    \begin{align*}
        \left|\cR^{(i)}(r,A)\right|_{g^i_{-1}}\le \left|B_{\fg}(x_0,A+C(n))\cap\mathcal{R}_X\right|+\Psi(i^{-1}\,|\,n,A,r)\quad \text{ for all }\ i,
    \end{align*}
    where $\Psi(i^{-1}\,|\,r,A)\to0$ while $i\to \infty$ and $r$, $A$ are fixed.
\end{Lemma}
\begin{proof}
Let 
\begin{align}\label{eq:limiting diffeo}
    \psi_i:\mathcal{R}\supset U_i\to V_i\subset M^i\times (-T_i,0]
\end{align}
be the sets and diffeomorphism defined in Theorem \ref{Thm_smooth_convergence}. By the smooth convergence on $\mathcal{R}$, it is obviously sufficient to show that
\begin{align*}
    M^i\times\{-1\}\supset \psi_i\left( U_i\cap B_{\fg}(x_0,A+C(n))\right)\supset \cR^{(i)}(r,A)\times\{-1\}
\end{align*}
for all $i$ large enough, depending on $r$ and $A$.

    For each $i$, we apply a covering argument to find a finite sequence of points $\left\{x^{(i)}_k\right\}_{k=1}^{N_i}$ in $\cR^{(i)}(r,A)$ such that 
    \begin{align*}
        \cR^{(i)}(r,A)\subset \bigcup_{k=1}^{N_i} B_{g^i_{-1}}\left(x_k^{(i)},r/3\right),\,\,N_i\le \overline{N}(n,A,r).
    \end{align*}
This is possible since for each $x\in \cR^{(i)}(r,A)$, we have $|B_{g^i_{-1}}(x,r/6)|_{g^i_{-1}}\ge c(Y)r^n$ by \cite[Theorem 6.1]{Bam20a}, and $|\cR^{(i)}(r,A)|_{g^i_{-1}}\le |B_{g^i_{-1}}(z_i,A)|_{g^i_{-1}}\le C(Y)A^n$ by \cite[Theorem 8.1]{Bam20a}. After passing to a (not relabeled) subsequence, we may assume that $N_i\equiv N \le \overline{N}(n,A,r)$ for all $i$. Next, we show that, for each $k\in\{1,2,\hdots,N\}$, the sequence $\left\{x_k^{(i)}\right\}_{i=1}^{\infty}$ converges, after passing to a subsequence, to some point in $\cR_{X}$. 
    \\

\noindent\textbf{Claim. } \emph{
For each $k\in\{1,2,\hdots,N\}$, there is a point $\tilde{x}_k\in B_{\fg}\lt(x_0,A+C(n)\rt)\subset\mathcal{R}_X=\mathcal{R}_{-1}$, such that, after passing to a subsequence,}   
\begin{align}\label{eq:points converging smooth}
    \psi_i^{-1}\left(x^{(i)}_k,-1\right)\xrightarrow{i\to\infty} \tilde{x}_k.
\end{align}
\emph{Furthermore, $P\left(\tilde x_k;\frac{1}{2}r\right)$ is unscathed and}
\begin{align}\label{eq:cube converging smooth}
   \psi_i\lt( P\lt(\tilde x_k;r/2 \rt)\rt)\supset P\left( x^{(i)}_k,-1;r/3\right)\quad\textit{for $i$ large enough}.
\end{align}

\begin{proof}[Proof of the claim]
    For the sake of notational simplicity, we shall denote $x^{(i)}_k$ by $x^{(i)}$ and $\tilde{x}_k$ by $\tilde{x}$, and suppress the subindex.  Let $0<r'\ll r$ be a small number to be fixed, and set $t'=-1-r'^2.$ Since $x^{(i)}\in B_{g^i_{-1}}(z_i,A)$, we have that
    \begin{align*}
        d^{g^i_{-1}}_{W_1}\left(\delta_{x^{(i)}}, \nu^i_{x_i,0\,;\,-1}\right) 
        &\le d^{g^i_{-1}}_{W_1}\left(\delta_{x^{(i)}}, \delta_{z_i}\right) + d^{g^i_{-1}}_{W_1}\left(\delta_{z_i}, \nu^i_{x_i,0\,;\,-1}\right)  \\
        &\le A+\sqrt{H_n},
    \end{align*}
    Thus, by \cite[Theorem 6.20]{Bam23}, we may find a conjugate heat flow $(\tilde{\nu}_t)_{t<-1}$ on the metric soliton $(\mathcal{X},\nu_t)$, such that
    \begin{equation*}
        \left(\nu^i_{x^{(i)},-1\,;\,t}\right)_{t\le -1}\xrightarrow[i\to\infty]{\fC} (\tilde{\nu}_t)_{t<-1},\qquad \lim_{t\nearrow -1}\Var(\tilde{\nu}_t)=0,
    \end{equation*}
   where $\mathfrak{C}$ is the correspondence in \eqref{def-limiting-condition}. Since the limit metric flow pair is also $H_n$-concentrated \cite[Theorem 7.4]{Bam23}, by \cite[Definition 2.6, Proposition 3.23]{Bam23}, we have
    \begin{align*}
        \iint_{\mathcal{X}_t\times\mathcal{X}_t}d_t^2(x,y)\,d\tilde\nu_t(x)d\tilde\nu_t(y) =\Var(\tilde\nu_t)\le H_n|t+1|\quad \text{ for all }\quad t<-1
    \end{align*}
    In particular, we can find a $y'\in\mathcal{X}_{t'}$, such that
    \begin{align*}
        \Var\lt(\delta_{y'},\tilde{\nu}_{t'}\rt)=\int_{\mathcal{X}_{t'}}d_{t'}^2(x,y')\,d\tilde\nu_{t'}(x) \le \iint_{\mathcal{X}_{t'}\times\mathcal{X}_{t'}}d_{t'}^2(x,y)\,d\tilde\nu_{t'}(x)d\tilde\nu_{t'}(y) \le H_n|t'+1|.
    \end{align*}
    Since $\mathcal{R}_{t'}$ is dense in $\mathcal{X}_{t'}$, we may without loss of generality assume $y'\in \mathcal{R}_{t'}$, the error caused by which can be made as small as we like. We then have, by \cite[Lemma 2.8]{Bam23} and $t'=-1-r'^2$,
    \begin{equation}\label{ineqHn12}
        d_{W_1}^{\mathcal{X}_{t'}}(\delta_{y'},\tilde\nu_{t'})\le\sqrt{\Var(\delta_{y'},\tilde\nu_{t'})}\le C(n)r'.
    \end{equation}
Since $y'\in\mathcal{R}_{t'}\subset \mathcal{R}$, for all $i$ large enough, we may set
    $$(y'_i,t'):=\psi_i(y')\in M^i\times\{t'\},$$
    according to Theorem \ref{Thm_smooth_convergence}, where $\psi_i$ is the diffeomorphism in \eqref{eq:limiting diffeo}.
    
    Next, we estimate the space-time distance between $\left(x^{(i)},-1\right)$ and $(y'_i,t').$ Since the convergence \eqref{def-limiting-condition} is time-wise at $t'$, we have (see \cite[Definition 6.7]{Bam23})
    \begin{equation}\label{ineqdW1Z1}
        d_{W_1}^{Z_{t'}}\left((\varphi_{t'}^i)_*\nu_{x^{(i)},-1;t'}^i,(\varphi_{t'})_*\tilde\nu_{t'}\right)\to0.
    \end{equation}
$\fC=(Z_t,\varphi^i_t,\varphi_t)$ is the correspondence, $\varphi_{t}^i:(M^i,g^i_t)\to (Z_t,d_t^Z)$ and $\varphi_t:(\cX_t,d_t)\to (Z_t,d^Z_t)$ are isometric embeddings. Putting (\ref{ineqHn12}) into $\fC$ yields
    \begin{equation}\label{ineqdW1Z2}
        d_{W_1}^{Z_{t'}}\left( (\varphi_{t'})_*\delta_{y'} ,(\varphi_{t'})_*\tilde\nu_{t'} \right)\le C(n)r'.
    \end{equation}
  The definition of $(y_i',t')$ implies that (see Theorem \ref{Thm_smooth_convergence} (c))
    \begin{equation}\label{ineqdW1Z3}
        d_{W_1}^{Z_{t'}}\left((\varphi_{t'}^i)_*\delta_{y_i'},(\varphi_{t'})_*\delta_{y'}\right)\to0.
    \end{equation}
    By (\ref{ineqdW1Z1})-(\ref{ineqdW1Z3}) and the triangle inequality, we have
        $d_{W_1}^{Z_{t'}}\lt( (\varphi_{t'}^i)_*\nu_{x^{(i)},-1;t'}^i,  (\varphi_{t'}^i)_*\delta_{y_i'}  \rt) \le C(n)r'$ for all $i$ large enough. This is then equivalent to saying that for $i$ large enough,
    \begin{equation}\label{ineq not spacetime}
        d_{W_1}^{g^i_{t'}}\lt(\nu^i_{x^{(i)},-1;t'},\delta_{y_i'}\rt)\le C(n)r'.
    \end{equation}

    By the assumption $x^{(i)}\in\cR^{(i)}\lt(r,A\rt)$,  we have $r_{\Rm}\lt(x^{(i)},-1\rt)\ge r$. Since $r'\ll r$, we certainly have  $|\Ric|\le r'^{-2}$ on $P\lt(x^{(i)},-1;r'\rt)$. Thus, by \cite[Proposition 9.5]{Bam20a},
    $d_{W_1}^{g^i_{t'}}\lt(\nu^i_{x^{(i)},-1;t'},\delta_{x^{(i)}}\rt)\le C(n)r'$. 
    In combination with (\ref{ineq not spacetime}), we have
    \begin{align*}
       d_{g^i_{t'}}\lt(x^{(i)},y_i'\rt)= d_{W_1}^{g^i_{t'}}\lt(\delta_{x^{(i)}},\delta_{y_i'}\rt)\le d_{W_1}^{g^i_{t'}}\lt(\nu^i_{x^{(i)},-1;t'},\delta_{x^{(i)}}\rt)+d_{W_1}^{g^i_{t'}}\lt(\nu^i_{x^{(i)},-1;t'},\delta_{y_i'}\rt)      \le  C(n)r'.
    \end{align*}
    A standard distance distortion estimate implies that
    \begin{align*}
       d_{g^i_{-1}}\lt(x^{(i)},y_i'\rt)\le C(n)r'.
    \end{align*}
     Thus, we have $(y_i',t')\in P\left(x^{(i)},-1;C(n)r'\right).$ 

    Now we fix $r'=r'(r)\ll r$ to ensure that, via the standard distance distortion estimate, 
    \begin{align*}
        r_{\Rm}\lt(y_i',t'\rt)\ge 0.99r\quad \text{ and }\quad (x^{(i)},-1)\in P\left(y_i',t';0.01r\right).
    \end{align*}
    By Theorem \ref{Thm_smooth_convergence}, the parabolic cube $P\left(y';0.98r\right)\Subset\mathcal{R}$ is unscathed. Note that no world-line in the smooth part of a metric soliton can become extinct in time, since the world-lines are simply integral curves of the complete field $\nabla f_0$. By the smooth convergence on $P\left(y';0.98r\right)$, we then have
    \begin{align*}
        \psi_i^{-1}\left(x^{(i)},-1\right)\in P\left(y';0.98r\right)\quad \text{ for all }\ i\ \text{ large enough},
    \end{align*}
    and then, after passing to a sequence, we have
    \begin{align*}
        \psi^{-1}_i\left(x^{(i)},-1\right)\to \tilde{x}\in P\left(y';0.02r\right)\cap\mathcal{R}_{-1}.
    \end{align*}
    If we take $r'$ to be even smaller if necessary (depending only on $n$ and $r$), then the standard distance distortion estimate implies that
    $$P\left(\tilde x;r/2\right)\subset P\left(y';0.98r\right),$$
    and \eqref{eq:points converging smooth}\eqref{eq:cube converging smooth} follow easily by the smooth convergence on $P\left(\tilde x;r/2\right)\Subset \mathcal{R}_{-1}$.

To see $\tilde{x}\in B_{\fg}\lt(x_0,A+C(n)\rt)$, we compute
\begin{align*}
    d_{\fg}(x_0,\tilde{x})&\le \ d^Z_{-1} \lt(\varphi_{-1}(\tilde{x}),\varphi^i_{-1}(x^{(i)})\rt)
    +d_{g^i_{-1}}\lt(x^{(i)},z_i\rt) + d^Z_{-1}\lt(\varphi_{-1}(x_0),\varphi^i_{-1}(z_i)\rt)
    \\
    &= d^{Z}_{-1}\lt( \varphi_{-1} \lt(\tilde{x} \rt),\varphi^i_{-1}(x^{(i)})\rt) + d_{g^i_{-1}}\lt(x^{(i)},z_i\rt)+d^{\fg}_{W_1}\left(\delta_{x_0},\nu_{-1}\right)\\
    &\quad + d^{Z_{-1}}_{W_1}\left((\varphi_{-1})_*(\nu_{-1}),(\varphi^i_{-1})_* \lt (\nu^i_{x_i,0\,;\,-1} \rt)\right)+ d^{g^i_{-1}}_{W_1}\left(\nu^i_{x_i,0\,;\,-1},\dt_{z_i}\right).
\end{align*}
Since the convergence \eqref{def-limiting-condition} is time-wise at $t=-1$, taking a limit at the right-hand-side of the above inequality, we conclude from Theorem \ref{Thm_smooth_convergence} (c) that
$$d_{\fg}(x_0,\tilde x)\le C(n)+A.$$
\end{proof}


We continue with the proof of the lemma. Applying the claim successively to $\{x^{(i)}_k\}_{i=1}^\infty$ for $k=1,2,\hdots,N$, we find $\tilde x_k\in B_{\fg}(x_0,A+C(n))\cap\mathcal{R}_X$, $k=1,2,\hdots,N$, such that after passing to subsequences for finitely many times, \eqref{eq:points converging smooth} and \eqref{eq:cube converging smooth} both hold. Thus, if $r<1$ and $A\ge \underline{A}(n)$, we have
\begin{align*}
       K:=\left( \bigcup_{k=1}^N P\left(\tilde{x}_k;r/2\right)\right)\cap\mathcal{R}_{-1}\Subset B_{\fg}(x_0,A+1+C(n))\cap\cR_X,
\end{align*}
for all $i$ large enough. Since 
\begin{align*}
    \psi_i \left( B_{\fg} \lt(x_0,A+C(n)\rt)\cap\cR_X\right)&\supset
    \psi_i(K)\supset \left(\bigcup_{k=1}^NP\left(x^{(i)}_k,-1;r/3\right)\right)\cap M\times\{-1\} \\
    &= \bigcup_{k=1}^NB_{g^i_{-1}}\left(x^{(i)}_k,r/3\right)\supset\mathcal{R}^{(i)}(r,A), 
\end{align*}
the conclusion follows from the smooth convergence on $K\Subset\mathcal{R}_{-1}$.  
\end{proof}

\subsection{A distance distortion estimate}

We show a distance distortion estimate on Ricci flows with a type-I scalar curvature bound. The proof of the following Lemma is similar to \cite[Proposition 2.21]{FL25}.
\begin{Lemma}\label{lem distance distortion}
    Let $\lt(M^n,g_t\rt)_{t\in I}$ be a Ricci flow with bounded curvature within each compact time-interval and suppose that the following hold:
    \begin{enumerate}[(a)]
        \item $x,y\in M$, and $[t_0-T,t_0]\subset I$.
        \item $R_{\min}\le R\le\dfrac{R_0}{t_0-t}$ on $\{x,y\}\times[t_0-T,t_0]$.
        \item $\cN_{x,t_0}(\tau)$, $\cN_{y,t_0}(\tau)\ge-Y$ for all $\tau\in[0,T]$.
    \end{enumerate}
Then for all $t\in[t_0-T,t_0]$, we have
    \begin{align*}
        d_t(x,y)\le d_{t_0}(x,y)+C(n,Y,R_0,R_{\min})|t_0-t|^{1/2}.
    \end{align*}
\end{Lemma}
\begin{proof}
    We may without loss of generality assume that $t_0=0$. Consider the potential function of the heat kernel  and Perelman's reduced length $\ell$ based at $(x,0)$, and we apply \cite[Corollary 9.5]{Per02} to obtain 
    \begin{align*}
        f(x,t)\le \ell(x,t) \le \frac{1}{2\sqrt{|t|}}\int_0^{|t|} \sqrt{\tau}R(x,-\tau)\,d\tau \le R_0, \,\,\text{for all}\,\,t\in[-T,0].
    \end{align*}
    Let $(z,t)$ be an $H_n$-center of $(x,0)$. We apply \cite[Theorem 7.2]{Bam20a} to the heat kernel $K(x,0;x,t)$, and this implies
    \begin{align}\label{eq:Hn-center drifting}
        d_t(x,z)&\le 3\lt(\ell(x,t)+C(n,R_{\min})-\cN_{x,0}\lt(|t|\rt)\rt)^{1/2}|t|^{1/2}\le C(n,Y,R_0,R_{\min})|t|^{1/2}.
    \end{align}
    On the other hand, by the definition of $H_n$-center, we have that
    \begin{align*}
        d_{W_1}^{g_t}(\nu_{x,0;t},\delta_z)\le\sqrt{H_n|t|}.
    \end{align*}
    Let $(z',t)$ be an $H_n$-center of $(y,0)$. A similar discussion yields
    \begin{align*}
        d_t(z',y)\le C(n,Y,R_0,R_{\min})|t|^{1/2},\,\,d_{W_1}^{g_t}(\nu_{y,0;t},\dt_{z'})\le\sqrt{H_n|t|}.
    \end{align*}
    Combining the above inequalities, and by the monotonicity of the $W_1$-Wasserstein distance, we finish the proof by computing
    \begin{align*}
        d_t(x,y)&\le d_t(x,z)+d_t(z,z')+d_t(z',y)\\
        &\le2C(n,Y,R_0,R_{\min})|t|^{1/2}+d_{W_1}^{g_{t}}(\nu_{x,0;t},\nu_{y,0;t})+d_{W_1}^{g_t}(\nu_{x,0;t},\dt_{z})+d_{W_1}^{g_t}(\nu_{y,0;t},\dt_{z'})\\
        &\le2C(n,Y,R_0,R_{\min})|t|^{1/2}+d_0(x,y)+2H_n^{1/2}|t|^{1/2}\\
        &\le C(n,Y,R_0,R_{\min})|t|^{1/2}+d_0(x,y).
    \end{align*}
\end{proof}

\subsection{An $\ep$-regularity theorem on volume ratio}
In this subsection, we prove Theorem \ref{thm main2}.

\begin{proof}[Proof of Theorem \ref{thm main2}]
    After parabolic rescaling and  time-shift, we may assume $r=3$ and $t_0=0$. We prove by contradiction. Assume that there exist a sequence of Ricci flows $\lt(M^i,g^i_t,x_i\rt)_{-9\le t\le0}$ and a sequence $\dt_i\searrow 0$  such that:
    \begin{enumerate}[(a)]
\item For all $s\in(0,3]$, we have $\lt|B_{g^i_0}(x_i,s)\rt|_{g^i_0}\ge\omega_ns^n(1-\dt_i)$;
\item $R^i\le\dfrac{R_0}{|t|}\,\,\text{on}\,\,B_{g^i_0}\lt(x_i,3\rt)\times[-9,0]$;
\item $ r^i_{\Rm}\lt(x_i,0\rt)<3\dt_i$;
\item $R^i\ge -\frac{n}{2}$ on $M^i\times[-8,0]$ by the standard maximum principle for the scalar curvature.
        \end{enumerate}

Without loss of generality, we may assume $\dt_i\le0.99$.   Using \cite[Corollary 5.11, Theorem 8.1]{Bam20a} and assumptions (a)(d), we get a uniform lower entropy bound depending only on the dimension
\begin{align}\label{ineq blowup N>}
    \inf_{x\in B_{g^i_0}(x_i,2)}\cN_{x,0}(\tau)\ge-C(n),\,\,\text{for all}\,\,\tau\in(0,4).
\end{align}
Next, set a sequence of factors 
$$\lambda_i:=\dt_i^{|\log\dt_i|^{-1/2}},$$
then it is clear that if $i\to+\infty$,
\begin{align}\label{ineq lambda dt converge}
    \lambda_i\to0,\,\,\text{and}\,\, \log_{\lambda_i}\dt_i\to+\infty.
\end{align}
By \eqref{ineq blowup N>}, we have an entropy gap 
\begin{align*}
    0\le \sum_{k=1}^{\dt_i^2\le\lambda_i^k} \lt( \cN_{x_i,0} \lt(4\lambda_i^{k+1} \rt)-\cN_{x_i,0} \lt(4\lambda_i^{k} \rt) \rt)\le C(n).
\end{align*}        
This implies we can find a $k_i\in \lt(0,2\log_{\lambda_i}\dt_i \rt]\cap\mathbb{N}$ such that the Nash entropy could be almost constant in combination with \eqref{ineq lambda dt converge}, namely,
\begin{align}\label{ineq cN gap}
    0\le \cN_{x_i,0}\lt(4\lambda_i^{k_i+1}\rt)-\cN_{x_i,0}\lt(4\lambda_i^{k_i}\rt)\le\frac{C(n)}{ \lt[2\log_{\lambda_i}\dt_i \rt]}\to0.
\end{align}

On the other hand, by assumption (b), we apply Lemma \ref{lem distance distortion} to deduce that for all $s\in(0,2]$ and $t\in[-4,0]$,
\begin{align*}
    B_{g^i_0}(x_i,s)\subset B_{g^i_t}\lt(x_i,C_d|t|^{1/2}+s\rt),
\end{align*}
where $C_d$ is a constant depending on $n$ and $R_0$, which may vary from line to line for at most finitely many times. In particular, $C_d$ is independent of $i$. The lower bound of the scalar curvature (assumption (d)) yields
\begin{align*}
    \frac{d}{dt}\lt|B_{g^i_0}(x_i,s)\rt|_{g^i_t}\le C(n) \lt|B_{g^i_0}(x_i,s) \rt|_{g^i_t}.
\end{align*}
Combining the two inequalities above, we have for all $s\in(0,2]$ and $t\in[-4,0]$,
\begin{align}\label{ineq volume>>}
    \lt|B_{g^i_t}\lt(x_i,s+C_d|t|^{1/2}\rt)\rt|_{g^i_t}\ge \lt|B_{g^i_0}(x_i,s)\rt|_{g^i_t}\ge e^{-C(n)|t|}\lt|B_{g^i_0}(x_i,s)\rt|_{g^i_0}\ge e^{-C(n)|t|}s^n\omega_n(1-\dt_i).
\end{align}

Now we set a sequence of blow-up factors
\begin{align*}
    r_i^2=\lambda_i^{k_i}\ge\dt_i^2
\end{align*}
and rescale each $\lt(M^i,g^i_t\rt)_{t\in[-4,0]}$ to get a sequence of Ricci flows:
\begin{align*}
   \lt(M^i,\tilde{g}^i_t\rt)_{t\in\lt[-4r_i^{-2},0\rt]},\,\, \tilde{g}^i_t=r_i^{-2}g^i\lt(r_i^2t\rt).
\end{align*}
Then for all $r\in \lt(C_d|t|^{1/2},2+C_d|t|^{1/2} \rt]$, and $t\in[-4,0]$, \eqref{ineq volume>>} becomes
\begin{align*}
    \frac{\lt|B_{\tilde{g}^i(r_i^{-2}t)} \lt(x_i,r_i^{-1}r\rt)\rt|_{\tilde{g}^i\lt(r_i^{-2}t\rt)}}{\omega_n \lt(r_i^{-1}r \rt)^n}\ge e^{-C(n)|t|}(1-\dt_i)\lt( \frac{r-C_d|t|^{1/2}}{r}\rt)^n.
\end{align*}
Set $t=-r_i^{2}$, we have for each $i$ and $A:=r_i^{-1}r\in \lt(C_d,C_d+2r_i^{-1}\rt]$,
\begin{align}\label{ineq rescaled AVR gap}
    \lt|B_{\tilde{g}^i_{-1}}(x_i,A)\rt|_{\tilde{g}^i_{-1}}\ge\omega_n e^{-C(n)r_i^2}(1-\dt_i)\lt(A-C_d\rt)^n.
\end{align}

Again \eqref{ineq cN gap} becomes 
\begin{align*}
    0\le \tilde{\cN}_{x_i,0}(4\lambda_i)-\tilde{\cN}_{x_i,0}(4)\le\frac{C(n)}{[2\log_{\lambda_i}\dt_i]}\to0.
\end{align*}
Hence, we apply Bamler's $\bF$-compactness theorem to the rescaled flows $\lt(M^i,\tilde{g}^i_t \rt)_{t\in\lt[-4r_i^{-2},0 \rt]}$ to get a metric soliton, namely,
\begin{align*}
    \lt( \lt(M^i,\tilde{g}^i_t \rt)_{t\in(-4,0)} ,\lt(\tilde{\nu}^i_{x_i,0;t}\rt)_{t\in(-4,0)} \rt)\xrightarrow[i\to\infty]{\mathbb{F},\mathfrak{C}}\lt(\mathcal{X},\lt(\nu_t\rt)_{t\in(-4,0)}  \rt).
\end{align*}
The $-1$ slice of $\lt(\mathcal{X},\lt(\nu_t\rt)_{t\in(-4,0)}  \rt)$ is the model of the limit metric soliton with regular part $\lt(\cR_X,\fg,f_0 \rt)$ and a center $x_0$. Let $\lt(z_i,-1 \rt)$ be an $H_n$-center of $(x_i,0)$ with respect to the flow $\tilde{g}^i_t$. Then similar to the proof of \eqref{eq:Hn-center drifting}, we get
\begin{align*}
    d_{\tilde{g}^i_{-1}}(x_i,z_i)\le C_d\quad \text{ for all }\quad i.
\end{align*}
Thus, for each $i$ and $A\in \lt(2C_d,2C_d+2r_i^{-1} \rt]$, \eqref{ineq rescaled AVR gap} becomes
\begin{align}\label{ineq rescaled AVR gap'}
    \lt|B_{\tilde{g}^i_{-1}}\lt(z_i,A\rt)\rt|_{\tilde{g}^i_{-1}}\ge \omega_n e^{-C(n)r_i^2}\lt(A-2C_d\rt)^n(1-\dt_i).
\end{align}

Applying Lemma \ref{Lm: volume of the singular part} and Lemma \ref{Lm:smooth volume convergence} yields that for $A\in[\underline{A}(n),+\infty)$ and $r\le\bar{r}(A,n)$,
\begin{align*}
    \lt|B_{\tilde{g}^i_{-1}}(z_i,A)\rt|_{\tilde{g}^i_{-1}}\le \lt|B_{\fg}\lt(x_0,A+C(n)\rt)\cap\cR_X \rt|+C(A,n)r^{\alpha}+\Psi(i^{-1}\,|\,n,A,r).
\end{align*}
Letting $i\to\infty$ and then $r\to0$,  we get from the above inequality together with \eqref{ineq rescaled AVR gap'} that
\begin{align*}
    \frac{\lt|B_{\fg}\lt(x_0,A\rt)\cap\cR_X \rt|}{\omega_n A^n}\ge \left(\frac{A-C(n)-2C_d}{A}\right)^n\quad \text{ for all }\quad A\ge \underline{A}(R_0,n).
\end{align*}
Now, if we take $A\ge\delta(n)^{-1}$ to be large enough, such that the right-hand-side of the above equation is greater than $1-\delta(n)$, where $\delta(n)$ is the constant in Theorem \ref{thm-AVR gap corollary}, then we immediately get that the limit metric soliton is actually Gaussian, which means $\tilde{r}_{\Rm}^i(x_i,-1)$ tends to infinity as $i\to\infty$. However, this contradicts assumption (c) that $\displaystyle \tilde{r}^i_{\Rm}\lt(x_i,0\rt)<3\dt_i/r_i\le3$ by Perelman's Pseudolocality Theorem \cite[Theorem 10.3]{Per02}.

\end{proof}

\subsection{An AVR gap theorem for ancient Ricci flows}
In this subsection we prove Theorem \ref{thm main}. The main difference from the proof of Theorem \ref{thm main2} is that we apply a blow-down technique to get a metric soliton.
\begin{proof}[Proof of Theorem \ref{thm main}]
Suppose $\lt(M^n,g_t\rt)_{t\in \lt(-\infty,0 \rt]}$ is an ancient Ricci flow with bounded curvature on compact time-intervals and type-I scalar curvature. Assume for some fixed base point $y_0\in M$ and some $\dt\in(0,1)$ to be determined later, we have
\begin{align}\label{ineq noncollapse}
   \limsup_{A\to+\infty}\frac{\lt|B_{g_0} \lt(y_0,A\rt)\rt|_{g_0}}{\omega_nA^n}\ge1-\dt. 
\end{align}
We first construct a sequence of blow-down factors to get an asymptotic volume gap of the $-1$ slice of the rescaled flow. Compared with the proof in the previous subsection, this case is much simpler because the scalar curvature is nonnegative due to \cite{Che09}.\\

\noindent\textbf{Claim.} \emph{For any $A\ge\underline{A}(n,\dt,R_0)\gg1$, we can find a sequence of blow-down factors $\tau_i\nearrow+\infty$, depending also on $A$, such that the rescaled flow $\lt(M,g^i_t:=\tau_i^{-1}g\lt(\tau_it\rt) \rt)_{t\le0}$ satisfies} 
\begin{equation*}
    \frac{\lt|B_{g^i_{-1}} \lt(y_0, A \rt)\rt|_{g^i_{-1}}}{\omega_n A ^n}\ge 1-3\dt,\,\,\text{for $i$ large enough}.
\end{equation*}

\begin{proof}[Proof of the claim]
 We may without loss of generality assume that $\dt\le0.99$. Then \cite[Corollary 5.11, Theorem 8.1]{Bam20a} and \eqref{ineq noncollapse} imply a uniform lower entropy bound
\begin{align}\label{ineq N>-Y}
    \inf_{\tau>0} \cN_{x,t}(\tau)\ge-Y_0, \,\,\text{for all}\,\,(x,t)\in M\times(-\infty,0],
\end{align}
where $Y_0$ is only a dimensional constant. In view of \eqref{ineq noncollapse},  we can find an increasing $\{\lambda_i\}$ with $\lambda_i\nearrow+\infty$ such that
\begin{equation}\label{ineq almost avr to 1}
    \frac{\lt|B_{g_0}\lt(y_0,\lambda_i^{1/2}\rt)\rt|_{g_0}}{\omega_n\lambda_i^{n/2}}\ge1-2\delta\quad\text{ for all }\ i.
\end{equation}

By \eqref{ineq N>-Y} and the type-I curvature assumption, we may apply Lemma \ref{lem distance distortion} to get
\begin{equation*}
    d_{g_{-\tau}}(x,y)\le d_{g_0}(x,y)+C_d\sqrt{\tau}\quad \text{ for all }\ x, y\in M,
\end{equation*}
where $C_d$, as before, is a constant depending on $n$ and $R_0$, which may vary from line to line. Choose a sequence of blow-down factors $\tau_i=\sigma\lambda_i$, where $\sigma<1$ is a fixed number to be determined later. 
Then the distance distortion estimate implies
\begin{equation*}
    B_{g_0}\lt(y_0,\lambda_i^{1/2}\rt)\subset B_{g_{-\tau_i}}\lt(y_0,\lambda_i^{1/2}\lt(1+C_d\sigma^{1/2}\rt)\rt).
\end{equation*}
Set $g^i_t:=\tau_i^{-1}g(\tau_it),t\le0$. 
Then by \eqref{ineq almost avr to 1}, we have,
\begin{align*}
    \frac{\lt|B_{g^i_{-1}} \lt( y_0,C_d+\sigma^{-1/2} \rt)\rt|_{g^i_{-1}}}{\omega_n \lt( C_d+\sigma^{-1/2} \rt)^n}&=\frac{\lt|B_{g_{-\tau_i}} \lt( y_0,\lambda_i^{1/2} \lt(1+C_d\sigma^{1/2} \rt) \rt)\rt|_{g_{-\tau_i}}}{\omega_n \lt( \lambda_i^{1/2} \lt(1+C_d\sigma^{1/2} \rt) \rt)^n }\\
    &\ge \frac{ \lt| B_{g_0}\lt(y_0,\lambda_i^{1/2}\rt) \rt|_{g_0}}{\omega_n \lambda_i^{n/2} }\frac{1}{\lt(1+C_d\sigma^{1/2}\rt)^n}\\
    &\ge\frac{1-2\delta}{\lt(1+C_d\sigma^{1/2}\rt)^n}=(1-2\dt) \lt(1- \frac{C_d}{C_d+\sigma^{-1/2}} \rt)^n,
\end{align*}
while the first inequality comes from the fact that the volume form is decreasing since $R\ge0$ \cite{Che09}.
Letting $C_d+\sigma^{-1/2}=A$, if $A\ge\underline{A}(n,R_0,\delta)$, then the right-hand-side of the equation is greater than $1-3\delta$; the proof is done. 

\end{proof}
Let $A\ge\underline{A}(n,R_0,\delta)$ be a constant to be determined later, and let $\tau_i$ be the scaling factors in the above claim, determined according to $A$. Define $g^i_t:=\tau_i^{-1}g_{\tau_it}$. Then, after passing to a subsequence, we have 
\begin{align*}
    \left( \lt(M,g^i_t\rt)_{t\in(-\infty,0]}, \lt(\nu^i_{y_0,0;t}\rt)_{t\in(-\infty,0]}\right)\xrightarrow[i\to\infty]{\mathbb{F},\mathfrak{C}}\lt(\mathcal{X},\lt(\nu_t\rt)_{t\in(-\infty,0)}\rt),
\end{align*}
where $\lt(\cX,\lt(\nu_t\rt)_{t\in(-\infty,0)}\rt)$ is a noncollapsed $\bF$-limit metric soliton called a tangent flow of $(M,g_t)_{t\le 0}$ at infinity. Suppose the $-1$ slice of $\cX$, also known as the model $(X,d,\nu)$ of the metric soliton , has regular part $(\cR,\fg,f_0)$ and a center $x_0$. 

By Lemma \ref{Lm: volume of the singular part} and Lemma \ref{Lm:smooth volume convergence}, we have for $A'\ge \underline{A'}(n)$,  $r\le\bar{r}(A')$,
\begin{align}\label{eq:volume estimate 112}
    \lt|B_{g^i_{-1}}\lt(z_i,A'\rt)\rt|_{g^i_{-1}}\le C(A')r^{\alpha}+\lt|B_{\fg}\lt(x_0,A'+C(n)\rt)\cap\cR_{X}\rt|+\Psi(i^{-1}\,|\,n,A',r),
\end{align}
where $(z_i,-1)$ is an $H_n$-center of $(y_0,0)$ with respect to the flow $g^i_t$.

On the other hand, by the same argument as \eqref{eq:Hn-center drifting}, we have
\begin{align*}
    d_{g^i_{-1}}(y_0,z_i)\le C_d.
\end{align*}
Thus, the claim implies that
\begin{align}\label{eq:volume estimate 113}
    \lt|B_{g^i_{-1}}\lt(z_i,A+C_d\rt)\rt|_{g^i_{-1}}\ge \omega_n A^n(1-3\delta).
\end{align}
Combining \eqref{eq:volume estimate 112} and \eqref{eq:volume estimate 113}, fixing $A'=A+C_d$, letting $i\to\infty$ first and then $r\to 0$, we get
\begin{align*}
    \lt|B_{\fg}\lt(x_0,A+C_d+C(n)\rt)\cap\cR_{X}\rt|\ge \omega_n A^n(1-3\delta),
\end{align*}
or equivalently
\begin{align*}
      \frac{\lt|B_{\fg}\lt(x_0,A+C_d+C(n)\rt)\cap\cR_{X}\rt|}{\omega_n(A+C_d+C(n))^n}\ge \left(\frac{A}{A+C_d+C(n)}\right)^n(1-3\delta).
\end{align*}
Thus, if we take $\delta=0.1\delta_{\ref{thm-AVR gap corollary}}$, where $\delta_{\ref{thm-AVR gap corollary}}$ is the dimensional constant in Theorem \ref{thm-AVR gap corollary}, and fix $A= A(n,R_0,\delta_{\ref{thm-AVR gap corollary}})\ge \delta_{\ref{thm-AVR gap corollary}}^{-1}$ large enough, such that the right-hand-side of the above equation is greater than $1-\delta_{\ref{thm-AVR gap corollary}}$,  then we obtain that the tangent flow is the Gaussian shrinker. Consequently, the ancient Ricci flow is the static Euclidean space. This finishes the proof.

\end{proof}

\bigskip


\begin{thebibliography}{CCG{\alphalchar{+}}10}










 















\bibitem[Bam20a]{Bam20a}
Bamler, Richard~H. \emph{Entropy and heat kernel bounds on a Ricci flow background}.
arXiv preprint arXiv:2008.07093 (2020).


  
\bibitem[Bam20b]{Bam20b} \bysame, \emph{Structure theory of non-collapsed limits of Ricci flows}. arXiv preprint  arXiv:2009.03243 (2020).


\bibitem[Bam23]{Bam23}\bysame, \emph{Compactness theory of the space of super Ricci flows.} Invent. Math. 233 (2023), no. 3, 1121–1277. 


 










\bibitem[CN09]{CN09} Carrillo, Jos\'e A., and Ni,Lei. \emph{Sharp logarithmic Sobolev inequalities on gradient solitons and applications.} Communications in Analysis and Geometry 17.4 (2009): 721-753.





\bibitem[CMZ25]{CMZ25} Chan, Pak-Yeung; Ma, Zilu; Zhang, Yongjia. \emph{ A local gap theorem for Ricci shrinkers.} Comm. Anal. Geom. 2025, to appear.




\bibitem[CMZ24]{CMZ24} Chan, Pak-Yeung; Ma, Zilu; Zhang, Yongjia. \emph{On noncollapsed $\bF$-limit of metric solitons.} arXiv preprint arXiv:2401.03387 (2023).



\bibitem[Che09]{Che09}Chen, Bing-Long,\emph{Strong uniqueness of the Ricci flow}. J. Differential Geom. 82 (2009), 363--382.











\bibitem[CM25]{CM25} Colding, Tobias Holck;  Minicozzi, William P. \emph{Singularities of Ricci flow and diffeomorphisms.} Publications mathématiques de l'IHÉS 142.1 (2025): 75-152.


 
\bibitem[CZ10]{CZ10}Cao, Huai-Dong; Zhou, De-Tang. \emph{On complete gradient shrinking Ricci solitons.} J. Differential Geom. \textbf{85} (2010), 175--186. 

\bibitem[CZ25]{CZ25}Cheng, Liang; Zhang, Yongjia. \emph{An $\ep$-regularity theorem for Perelman's reduced volume.} arXiv preprint 	arXiv:2502.15219

\bibitem[CZ26]{CZ26}Chan, Pak-Yeung; Zhang, Yongjia. \emph{Rigidity and gap theorems for Ricci shrinkers}  arXiv preprint arXiv: 2605.08884(2026).
 





\bibitem[DWZ26]{DWZ26} Deng, Yuxing; Wang, Ganqi; Zhang, Yongjia. \emph{Ancient Ricci flows with nonnegative Ricci curvature}, arXiv preprint arXiv:2603.28014. 

\bibitem[FL25]{FL25} Fang, Hanbing; Li, Yu. \emph{On the structure of noncollapsed Ricci flow limit spaces}, arXiv preprint arXiv:2510.12398(2025).












\bibitem[HN14]{HN14} Hein, Hans-Joachim; Naber, Aaron. \emph{New logarithmic Sobolev inequalities and an $\ep$-regularity theorem
for the Ricci flow}, Comm. Pure Appl. Math. 67 (2014), no. 9, 1543–1561. MR 3245102




\bibitem[Kot13]{Kot13} Kotschwar, Brett.\emph{A local version of Bando's theorem on the real-analyticity of solutions to the Ricci flow}. Bull. Lond. Math. Soc. 45 (2013), no.1, 153–158.





\bibitem[LW20]{LW20} Li, Yu; Wang, Bing. \emph{Heat kernel on Ricci shrinkers}. Calc. Var. Partial Differential Equations 59 (2020), no. 6, Paper No. 194, 84 pp.


\bibitem[LW24]{LW24} Li, Yu; Wang, Bing. \emph{Rigidity of the round cylinders in Ricci shrinkers}. Journal of Differential Geometry 127.2 (2024): 817-897.

\bibitem[LZ23]{LZ23} Li, Yu; Zhang, Wenjia. \emph{On the rigidity of Ricci shrinkers.} arXiv preprint arXiv:2305.06143 (2023).

\bibitem[LQZ26]{LQZ26} Li, Wenqi; Qu Yuanyuan; Zhu Meng. To appear


\bibitem[M20]{M20} Boris Mityagin, \emph{The Zero Set of a Real Analytic Function}. (Russian) Mat. Zametki 107 (2020), no.3, 473–475; translation in Math. Notes 107 (2020), no.3-4, 529–530.







\bibitem[Per02]{Per02} Perelman, Grisha, \emph{The entropy formula for the Ricci flow and its geometric applications}, arXiv:math.DG/0211159 (2002).



\bibitem[SS72]{SS72} Sou\v{c}ek, Ji\v{r}\'{\i}; Sou\v{c}ek, Vladim\'{\i}r. \emph{Morse-Sard theorem for real-analytic functions}. Commentationes Mathematicae Universitatis Carolinae, Vol. 13 (1972), No. 1, 45-51


\bibitem[WW25]{WW25} Wang, Jie; Wang, Youde. \emph{Rigidity and $\ep$-regularity theorems of Ricci shrinkers}, Calc. Var. 64, 42 (2025). https://doi.org/10.1007/s00526-024-02903-5



\bibitem[Yo12]{Yo12} Yokota, Takumi. \emph{Addendum to `Perelman’s reduced volume and a gap theorem for the Ricci flow'.} Communications in Analysis and Geometry 20.5 (2012): 949-955.


\bibitem[Zhang11]{Zhang11} Zhang, Shijin; \emph{On a sharp volume estimate for gradient Ricci solitons with scalar curvature bounded below,} Acta Math. Sin. (Engl. Ser.) 27(2011), no. 5, 871–882, DOI 10.1007/s10114011-9527-7. MR2786449

\bibitem[Zhang18]{Zhang18} Zhang, Shijin. \emph{A gap theorem on complete shrinking gradient Ricci solitons}. Proc. Amer. Math. Soc. 146 (2018), no. 1, 359–368.

\bibitem[Zhang20]{Zhang20} Zhang, Zhuhong. \emph{A gap theorem of four-dimensional gradient shrinking solitons}. Comm. Anal. Geom. 28 (2020), no. 3, 729–742

\bibitem[Zhang21]{Zhang21} Zhang, Yongjia; \emph{Entropy, noncollapsing, and a gap theorem for ancient solutions to the Ricci flow}, Communications in Analysis and Geometry, DOI:10.4310/
CAG.2021.v29.n2.a8



\end{thebibliography}
\bibliographystyle{amsalpha}

\newcommand{\alphalchar}[1]{$^{#1}$}
\providecommand{\bysame}{\leavevmode\hbox to3em{\hrulefill}\thinspace}
\providecommand{\MR}{\relax\ifhmode\unskip\space\fi MR }
\providecommand{\MRhref}[2]{%
  \href{http://www.ams.org/mathscinet-getitem?mr=#1}{#2}
}

\noindent School of Mathematics and Statistics, Beijing Institute of Technology, Beijing, 100081, China
\\ E-mail address: \verb"ggwgq5986@sina.com"
\\

\noindent School of Mathematical Sciences, Shanghai Jiao Tong University, Shanghai, 200240, China
\\ E-mail address: \verb"sunzhang91@sjtu.edu.cn"

\end{document}